\documentclass[12pt, reqno]{amsart}

\author[A.~Kwela]{Adam Kwela}
\address{Institute of Mathematics, Faculty of Mathematics, Physics and Informatics, University of Gda{\'n}sk, ul. Wita Stwosza 57, 80-308 Gda{\'n}sk, Poland} %
\email{adam.kwela@ug.edu.pl}

\author[P.~Leonetti]{Paolo Leonetti}
\address{Department of Decision Sciences, Universit\`a ``Luigi Bocconi'', via Roentgen 1, 20136 Milan, Italy}
\email{leonetti.paolo@gmail.com}
\urladdr{https://sites.google.com/site/leonettipaolo/}



\title{Density-Like and Generalized Density Ideals}

\usepackage[T1]{fontenc}
\usepackage{amsmath}
\usepackage{amssymb}
\usepackage{amsthm}
\usepackage[left=3.5cm, right=3.5cm, paperheight=11.8in]{geometry}
\usepackage{hyperref}
\usepackage{fancyhdr}
\usepackage{enumitem}
\usepackage{comment}
\usepackage{nicefrac}
\usepackage{bm}
\usepackage{mathrsfs}
\usepackage{graphicx}
\usepackage[utf8]{inputenc}
\usepackage{cancel}
\usepackage{mathtools}
\usepackage{tikz}
\usetikzlibrary{cd,decorations.pathreplacing,matrix,arrows,positioning,automata,shapes,shadows,calc,fadings,decorations,snakes,through,intersections}
\usepackage{float}

\AtBeginDocument{%
   \def\MR#1{}
}

\newtheorem{thm}{Theorem}[section]
\newtheorem{cor}[thm]{Corollary}
\newtheorem{lem}[thm]{Lemma}
\newtheorem{prop}[thm]{Proposition}

\theoremstyle{definition} 
\newtheorem{defi}[thm]{Definition}
\let\olddefi\defi
\renewcommand{\defi}{\olddefi\normalfont}
\newtheorem{question}{Question}
\let\oldquestion\question
\renewcommand{\question}{\oldquestion\normalfont}
\newtheorem{example}[thm]{Example}
\let\oldexample\example
\renewcommand{\example}{\oldexample\normalfont}
\newtheorem{rmk}[thm]{Remark}
\let\oldrmk\rmk
\renewcommand{\rmk}{\oldrmk\normalfont}

\newcommand{\Exh}{\mathrm{Exh}}
\newcommand{\Fin}{\mathrm{Fin}}
\newcommand{\I}{\mathcal{I}}
\newcommand{\Dweak}{\mathrm{D}_{\mathrm{weak}}}
\newcommand{\Dstrong}{\mathrm{D}_{\mathrm{strong}}}

\hypersetup{
    pdftitle={Density-Like and Generalized Density Ideals},
    pdfauthor={Adam Kwela and Paolo Leonetti},
    pdfmenubar=false,
    pdffitwindow=true,
    pdfstartview=FitH,
    colorlinks=true,
    linkcolor=blue,
    citecolor=green,
    urlcolor=cyan
}

\providecommand{\MR}[1]{}

\providecommand{\MR}{\relax\ifhmode\unskip\space\fi MR }

\providecommand{\href}[2]{#2}

\begin{document}

\maketitle
\thispagestyle{empty}

\begin{abstract}
We show that there exist uncountably many 
(tall and nontall) 
pairwise
nonisomorphic density-like ideals on $\omega$ which are not generalized density ideals. 
In addition, they are nonpathological. 
This answers a question posed 
by Borodulin-Nadzieja, Farkas, and Plebanek 
in [J. Symb. Log. \textbf{80} (2015), 1268--1289]. 
Lastly, we provide sufficient conditions for a density-like ideal to be necessarily a generalized density ideal.
\end{abstract}


\section{Introduction}\label{sec:intro}

An ideal $\mathcal{I}$ on the nonnegative integers $\omega$ is a family of subsets of $\omega$ closed under finite unions and subsets. 
Unless otherwise stated, we assume that $\mathcal{I}$ is admissible (i.e., it contains $\mathrm{Fin}:=[\omega]^{<\omega}$) and proper (i.e., $\omega\notin \mathcal{I}$). 
An ideal $\mathcal{I}$ is tall if each infinite set $A\subseteq \omega$ contains an infinite subset in $\mathcal{I}$. 
It is a P-ideal if it is $\sigma$-directed modulo finite sets, i.e., for each sequence $(A_n)$ in $\mathcal{I}$ there is $A\in\mathcal{I}$ such that $A_n\setminus A$ is finite for all $n$.
Ideals are regarded as subsets of $\{0,1\}^\omega$ with the Cantor-space topology, hence it is possible to speak about Borel, analytic ideals, etc. 
We refer to \cite{MR2777744, MR3920747} for recent surveys on ideals and associated filters.

A lower semicontinuous submeasure (lscsm) $\varphi:\mathcal{P}(\omega) \to [0,\infty]$ is a subadditive monotone function such that $\varphi(\emptyset)=0$, $\varphi(\{n\})<\infty$ for all $n\in\omega$, and $\varphi(A)=\lim_n \varphi(A\cap n)$ for all $A\subseteq \omega$ (here, as usual, each $n$ is identified with $\{0,1,\ldots,n-1\}$).  
Denote by $\mathrm{supp}(\varphi):=\{n \in \omega: \varphi(\{n\})\neq 0\}$ its support. 
A lscsm with finite support will be typically denoted by $\mu$. 
It is folklore that the pointwise supremum of lscsms is a lscsm. 
For each lscsm $\varphi$, we associate its exhaustive ideal 
$$
\mathrm{Exh}(\varphi):=\left\{A\subseteq \omega: \|A\|_\varphi=0\right\},
$$
where $\|A\|_\varphi:=
\inf_{F \in \mathrm{Fin}} \varphi(A\setminus F)$. A classical result of Solecki states that a (not necessarily proper) ideal $\mathcal{I}$ is an analytic P-ideal if and only if $\mathcal{I}=\mathrm{Exh}(\varphi)$ for some lscsm $\varphi$ such that $\varphi(\omega)<\infty$, 
 see e.g. \cite[Section 1.2]{MR1711328} for a textbook exposition. In particular, each analytic P-ideal is $F_{\sigma\delta}$. 
Every lscsm $\varphi$ defines a metric $d_\varphi$ on $\mathcal{I}=\mathrm{Exh}(\varphi)$ given by $d_\varphi(A,B)=\varphi(A\triangle B)$ for all $A,B \in \mathcal{I}$. 
The topology induced on $\mathcal{I}$ is Polish and does not depend on the choice of $\varphi$, see \cite{MR1708146}.
 
The aim of this work is to study the relationship between two families of analytic P-ideals defined below.

\begin{defi}\label{def:generalizeddensityideal}
An ideal $\mathcal{I}$ is called a \textbf{generalized density ideal} if there exists a sequence $\bm{\mu}=(\mu_n)$ of lscsms with finite pairwise disjoint supports such that $\mathcal{I}=\mathrm{Exh}(\varphi_{\bm{\mu}})$, where $\varphi_{\bm{\mu}}:=\sup_n \mu_n$. 
\end{defi}
Note that if $\bm{\mu}=(\mu_n)$ is a sequence of lscsms with finite pairwise disjoint supports, then
$$
\mathrm{Exh}\left(\varphi_{\bm{\mu}}\right)=\mathrm{Exh}\left(\limsup_{n\to \infty}\mu_n\right).
$$

Generalized density ideals have been introduced by Farah in \cite[Section 2.10]{MR1988247}, see also \cite{MR2254542}, and have been used in different contexts, see e.g. \cite{MR3436368, MR2320288, MR3968131}. 
We remark that Farah's original definition assumed that $\{\mathrm{supp}(\mu_n): n \in \omega\}$ is a partition of $\omega$ into finite intervals; however, we will show in Proposition \ref{prop:supportgeneralizeddensityideal} that this is equivalent to Definition \ref{def:generalizeddensityideal}. 
The family of generalized density ideals is very rich. Indeed, if each $\mu_n$ is a measure then $\mathrm{Exh}(\varphi_{\bm{\mu}})$ is a \textbf{density ideal}, as defined in \cite[Section 1.13]{MR1711328}, cf. also \cite[Proposition 6.3]{MR3436368}. In particular, it includes $\emptyset \times \mathrm{Fin}$, the ideal of density zero sets
$$
\mathcal{Z}:=\left\{A\subseteq \omega: \lim_{n\to \infty}\frac{|A\cap n|}{n}=0\right\},
$$
and all the Erd{\H o}s--Ulam ideals introduced by Just and Krawczyk in \cite{MR748847}, that is, ideals of the type $\mathrm{Exh}(\varphi_f)$ where $f: \omega \to (0,\infty)$ is a function such that $\sum_{n}f(n)=\infty$, $f(n)=o\left(\sum_{i\le n}f(i)\right)$ as $n\to \infty$, and $\varphi_f: \mathcal{P}(\omega) \to (0,\infty)$ is the submeasure defined by
$$
\forall A\subseteq \omega,\quad \varphi_f(A)=\sup_{n \in \omega} \frac{\sum_{i\le n,\, i \in A}f(i)}{\sum_{i\le n}f(i)},
$$
see \cite[pp. 42--43]{MR1711328}.  
In addition, this family contains the ideals associated with suitable modifications of the natural asymptotic density, the so-called \textbf{simple density ideals}, see \cite{MR3391516, MR3950052} and Section \ref{sec:nontall} below.  
Lastly, a large class of generalized density ideals has been defined by Louveau and Veli\v{c}kovi\'{c} in \cite{MR1708151, MR1169042}, cf. also \cite[Section 2.11]{MR1988247}.

\begin{defi}\label{def:densitylike}
An ideal $\mathcal{I}$ is said to be \textbf{density-like} if $\mathcal{I}=\mathrm{Exh}(\varphi)$ for a density-like lscsm $\varphi$, that is, a lscsm such that for all $\varepsilon>0$ there exists $\delta>0$ for which, if $(F_n) \in \mathrm{Fin}^\omega$ is a sequence of finite pairwise disjoint sets with $\varphi(F_n)<\delta$ for all $n$, then $\varphi(\bigcup_{i \in I}F_i) < \varepsilon$ for some infinite $I\subseteq \omega$. 
\end{defi}

The class of density-like ideals played a role in \cite{Matrai, MR2787694}. The main result of \cite{MR2787694} states that the ideal $\mathrm{NWD}$ of all closed nowhere dense subsets of $2^\omega$ is not Tukey reducible to any density-like ideal $\I$ (that is, there is no function $f:\mathrm{NWD}\to\I$ such that for each $A\in\I$ there exists $B\in\mathrm{NWD}$ for which $f(X)\subseteq A$ implies $X\subseteq B$, i.e., preimages of bounded sets are bounded). 
In particular, this works for $\mathcal{Z}$ (since it is a density-like ideal), thus answering 
old questions of Isbell from 1972 and Fremlin from 1991.

It is known that if $\varphi$ is density-like and $\mathrm{Exh}(\varphi)=\mathrm{Exh}(\psi)$, for some lscsm $\psi$, then $\psi$ is density-like too. In addition, tall $F_\sigma$ P-ideals are not density-like, see \cite[Fact 5.1]{MR3436368}, and there exists a nontall $F_\sigma$ P-ideal which is not density-like, see \cite{Matrai}. 

On the one hand, every generalized density ideal $\mathcal{I}=\mathrm{Exh}(\varphi_{\bm{\mu}})$ is a density-like ideal (indeed, $\varphi_{\bm{\mu}}$ is a density-like lscsm).  On the other hand, the converse has been asked in \cite[Question 5.11]{MR3436368}: 
\begin{question}\label{q:questiondensitylike}
Is there a density-like ideal which is not a generalized density ideal?
\end{question}

This question is closely connected to the notion of representability of ideals in Polish Abelian groups and in Banach spaces. Following \cite{MR3436368}, we say that an ideal $\I$ on $\omega$ is representable in a Polish Abelian group $X$ if there is a function $f:\omega\to X$ such that 
$$
\textstyle 
A\in\I\ \Longleftrightarrow\ \sum_{n\in A}f(n)\text{ is unconditionally convergent in }X.
$$
By \cite[Theorems 4.1 and 4.4]{MR3436368}, an ideal is representable in some Polish Abelian group if and only if it is an analytic P-ideal, and it is representable in some Banach space if and only if it is a nonpathological analytic P-ideal (cf. Remark \ref{rmk:nonpathological}). Moreover, for instance, it is known that an ideal is representable in $\mathbb{R}^\omega$ if and only if it is an intersection of countably many summable ideals \cite[Example 3.8]{MR3436368} (i.e., ideals of the form $\I_f:=\left\{A\subseteq\omega:\ \sum_{n\in A}f(n)<\infty\right\}$ for some $f:\omega\to[0,\infty)$ such that $\sum_{n}f(n)=\infty$); for more on this notion see \cite{MR3436368}. It is worth mentioning that P. Borodulin-Nadzieja and B. Farkas, using representability of ideals in Banach spaces, constructed a new example of a Banach space \cite[Example 5.9]{Wroclaw2}, and strengthened Mazur's Lemma \cite[Corollary 7.6]{Wroclaw2}, which is a basic tool in Banach space theory (they were able to specify the form of the convex combination in Mazur's Lemma). 
This suggests that studying the interplay  between 
representability 
and theory of analytic P-ideals may have some relevant yet unexploited potential for the study of the geometry of Banach spaces.

Question \ref{q:questiondensitylike} is motivated by the problem of characterizing ideals which are representable in the Banach space $c_0$ \cite[Question 5.10]{MR3436368}. It is known that a tall $F_\sigma$ P-ideal is representable in $c_0$ if and only if it is a summable ideal \cite[Theorem 5.7]{MR3436368} and that all nonpathological generalized density ideals are representable in $c_0$ \cite[Example 4.2]{MR3436368}. 

The motivation of this work is to shed some light on \cite[Question 5.10]{MR3436368} by providing a large class of density-like ideals which are not generalized density ideals. In particular, we give a positive answer to Question \ref{q:questiondensitylike}.
\begin{thm}\label{thm::main}
There exists a 
density-like ideal which is not a generalized density ideal.
\end{thm}
More precisely, 
our main contributions are:
\begin{enumerate}[label=(\roman*)]
\item There exist uncountably many nonpathological, nontall, and pairwise nonisomorphic density-like ideals which are not generalized density ideals, see Theorem \ref{cor:2omega}; 
\item There exist uncountably many nonpathological, tall, and pairwise nonisomorphic density-like ideals which are not generalized density ideals, see Theorem \ref{thm:mainsecondparttall}; 
\item A characterization of generalized density ideals which is reminiscent of the definition of density-like ideals, 
see Theorem \ref{char}.
\end{enumerate}

%



%


\section{Preliminaries}\label{sec:characterizations}

Given (not necessarily proper or admissible) ideals $\mathcal{I}, \mathcal{J}$ on $\omega$, we let their \emph{disjoint sum} and \emph{Fubini product} be 
$$
\mathcal{I} \oplus \mathcal{J}:=\left\{B\subseteq 2\times\omega: B_{(0)} \in \mathcal{I},B_{(1)} \in \mathcal{J} \right\},
$$
$$
\mathcal{I} \times \mathcal{J}:=\left\{B\subseteq \omega^2: \{m \in \omega: B_{(m)}\notin \mathcal{J}\} \in \mathcal{I} \right\},
$$
where $B_{(m)}:=\{k \in \omega: (m,k) \in B\}$. 
Then $\mathcal{I}\times \mathcal{J}$ is an ideal on $\omega^2$. 
We identify ideals on $\omega^2$ with ideals on $\omega$ through the bijection 
$h: \omega^2 \to \omega$ defined by 
\begin{equation}\label{eq:bijectionh}
\forall (x,y) \in \omega^2,\quad h(x,y):=2^{x}(2y+1)-1.
\end{equation}
To ease the notation, we define the families $\mathcal{F}_{\mathrm{disj}}$, $\mathcal{F}_{\mathrm{incr}}$, $\mathcal{F}_{\mathrm{int}}$ of sequences of nonempty finite sets which are, respectively, pairwise disjoint, increasing, and increasing intervals:
\begin{displaymath}
\begin{split}
\mathcal{F}_{\mathrm{disj}}:=&\{(F_n) \in (\mathrm{Fin}\setminus \{\emptyset\})^\omega: \forall \{i,j\} \in [\omega]^2, F_i \cap F_j=\emptyset\},\\
\mathcal{F}_{\mathrm{incr}}:=&\{(F_n) \in \mathcal{F}_{\mathrm{disj}}: \forall n \in \omega, \max F_n+1\le \min F_{n+1}\},\\
\mathcal{F}_{\mathrm{int}}:=&\{(F_n) \in \mathcal{F}_{\mathrm{incr}}: \forall n \in \omega, F_n \text{ is an interval}\}.
\end{split}
\end{displaymath}
%
%
In particular, $\mathcal{F}_{\mathrm{int}}\subseteq \mathcal{F}_{\mathrm{incr}}\subseteq \mathcal{F}_{\mathrm{disj}}$. 

We start with some characterizations of generalized density ideals, cf. also Proposition \ref{thm:casefsigma} and Theorem \ref{char} below.
\begin{prop}\label{prop:supportgeneralizeddensityideal}
Let $\mathcal{I}$ be an ideal. Then the following are equivalent:
\begin{enumerate}[label={\rm (\textsc{g}\arabic{*})}]
\item  \label{item:g1} $\mathcal{I}=\mathrm{Exh}(\varphi_{\bm{\mu}})$ for a sequence $\bm{\mu}=(\mu_n)$ of lscsms with $(\mathrm{supp}(\mu_n)) \in \mathcal{F}_{\mathrm{int}}$\textup{;}
\item \label{item:g2} $\mathcal{I}=\mathrm{Exh}(\varphi_{\bm{\mu}})$ for a sequence $\bm{\mu}=(\mu_n)$ of lscsms with $(\mathrm{supp}(\mu_n)) \in \mathcal{F}_{\mathrm{incr}}$\textup{;}
\item \label{item:g3} $\mathcal{I}=\mathrm{Exh}(\varphi_{\bm{\mu}})$ for a sequence $\bm{\mu}=(\mu_n)$ of lscsms with $(\mathrm{supp}(\mu_n)) \in \mathcal{F}_{\mathrm{disj}}$ \textup{(}that is, $\mathcal{I}$ is a generalized density ideal\textup{)}\textup{;}
\item  \label{item:g4} $\mathcal{I}=\mathrm{Exh}(\varphi_{\bm{\mu}})$ for a sequence $\bm{\mu}=(\mu_n)$ of bounded lscsms such that 
$$
\forall k \in \omega,\quad \{n \in \omega: k \in \mathrm{supp}(\mu_n)\} \in \mathrm{Fin}\textup{.}
$$ 
\end{enumerate}
\end{prop}
\begin{proof}
It is clear that \ref{item:g1} $\implies$ \ref{item:g2} $\implies$ \ref{item:g3} $\implies$ \ref{item:g4}. 

\ref{item:g4} $\implies$ \ref{item:g3} See \cite[Proposition 5.4]{MR3436368}. 

\ref{item:g3} $\implies$ \ref{item:g1} Suppose that $\mathcal{I}=\mathrm{Exh}(\varphi_{\bm{\mu}})$ for some sequence $\bm{\mu}=(\mu_n)$ of lscsms such that $(S_n) \in \mathcal{F}_{\mathrm{disj}}$, 
where $S_n:=\mathrm{supp}(\mu_n)$ for each $n$. Note that we can assume without loss of generality that $S:=\bigcup_n S_n=\omega$. Indeed, in the opposite, if $S^c$ is finite then it is sufficient to replace $\mu_0(A)$ with $\mu_0(A)+|A\cap S^c|$ for all $A\subseteq \omega$. Otherwise, let $(x_n)$ be the infinite increasing enumeration of $S^c$ and replace every $\mu_n(A)$ with $\mu_n(A)+\frac{1}{n}|A\cap \{x_n\}|$. This is possible, considering that 
$$
\textstyle \mathcal{I}=\mathrm{Exh}(\varphi_{\bm{\mu}})=\mathrm{Exh}(\limsup_{n}\mu_n).
$$


At this point, let $(T_n)\in \mathrm{Fin}^\omega$ be the sequence defined recursively as it follows: set $T_0:=[0,\max S_0]$ and, for each $n\in \omega$, set 
$$
\textstyle T_{n+1}:=\left(\max \bigcup_{i\le n}T_i,\, \max\left(S_{n+1}\cup \bigcup \{S_k: \min S_k \le \max \bigcup_{i\le n}T_i\}\right)\right].
$$
Observe that $(T_n)$ is a sequence of (possibly empty) pairwise disjoint finite intervals such that $\bigcup_{n}T_n=\omega$. Moreover, for each $n \in \omega$ there exists $j=j(n)\in \omega$ with $S_n\subseteq T_{j(n)}\cup T_{j(n)+1}$: indeed, if $j(n)$ is the minimal integer such that $S_n\cap T_{j(n)}\neq\emptyset$ (so that $T_{j(n)}\neq \emptyset$ and $\min S_n \le \max T_{j(n)}$), then 
$$
\textstyle \max(T_{j(n)+1})\geq\max\left(\bigcup\{S_k: \min S_k \le \max \bigcup_{i\le j(n)}T_i\}\right)\geq\max(S_n).
$$

Let $(V_n)$ be the biggest subsequence of $(T_n)$ with nonempty elements, so that $(V_n) \in \mathcal{F}_{\mathrm{int}}$, and define the sequence $\bm{\nu}=(\nu_n)$ of lscsms by
$$
\textstyle \forall n \in \omega, \forall A\subseteq \omega, \quad 
\nu_n(A):=\sup_k \mu_k(A \cap V_n)
$$
Note that $\nu_n(A)=\sup_{k\geq n}\mu_k(A\cap V_n)$, since $S_k\cap V_n=\emptyset$ whenever $k<n$ (indeed $S_k\subseteq \bigcup_{n\le k} T_n\subseteq \bigcup_{n\le k}V_n$ for all $k \in \omega$). Moreover, it follows by construction that $\text{supp}(\nu_n)=V_n$ for each $n \in \omega$, hence it is sufficient to show that 
$\mathcal{I}=\mathrm{Exh}(\varphi_{\bm{\nu}})$.

On the one hand, it is clear that if $\mu_n(A)\to 0$ then 
$$
\textstyle \nu_n(A)=\sup_{k\ge n}\mu_k(A\cap V_n) \le \sup_{k\ge n}\mu_k(A) \to 0,
$$
hence $\mathcal{I}\subseteq \mathrm{Exh}(\varphi_{\bm{\nu}})$. 

On the other hand, suppose that $\nu_n(A)\to 0$ and fix $\varepsilon>0$. Then there exists $n_0\in \omega$ such that $\nu_n(A)\le \nicefrac{\varepsilon}{2}$ for all $n> n_0$. Let $k_0$ be the minimal integer such that $S_k\cap\bigcup_{n\le n_0}V_n=\emptyset$ for all $k\ge k_0$. Fix $k\ge k_0$ and $n\in \omega$ such that $S_k\subseteq V_n\cup V_{n+1}$ (hence, in particular, $n>n_0$). We conclude that
\begin{displaymath}
\begin{split}
\mu_k(A)&=\mu_k(A\cap S_k)\leq\mu_k(A\cap (V_n\cup V_{n+1}))\\
&\textstyle \leq\mu_k(A\cap V_n)+\mu_k(A\cap V_{n+1}) \leq\nu_n(A)+\nu_{n+1}(A)\le \frac{\varepsilon}{2}+\frac{\varepsilon}{2}=\varepsilon,
\end{split}
\end{displaymath}
which shows that $\mu_k(A)\to 0$, therefore $\mathrm{Exh}(\varphi_{\bm{\nu}}) \subseteq \mathcal{I}$. 
\end{proof}

Some additional notations are in order. Given a lscsm $\varphi$ and a real $\delta>0$, let $\mathcal{G}_{\varphi,\delta}$ be the set of sequences of subsets of $\omega$ with $\varphi$-value smaller than $\delta$, that is, 
$$
\mathcal{G}_{\varphi,\delta}:=\left\{(F_n) \in \mathcal{P}(\omega)^\omega: \forall n\in \omega, \,\,\varphi(F_n)< \delta\right\}.
$$

Let $\mathcal{I}$, $\mathcal{J}$ be ideals on $\omega$ and let $\mathcal{I}^+$ be the family of $\mathcal{I}$-positive sets, that is, $\{A\subseteq \omega: A\notin \mathcal{I}\}$. Following Solecki and Todorcevic \cite[p. 1892]{MR2134228}, we say that a separable metric space $X$ is $(\mathcal{I}^+,\mathcal{J})$\emph{-calibrated} if the following property holds: for each sequence $x=(x_n)$ in $X$ with $\Gamma_x(\mathcal{I}) \neq \emptyset$, there exists $A\in \mathcal{I}^+$ such that $\{x_n: n \in A\cap B\}$ is bounded for all $B \in \mathcal{J}$ (where $\Gamma_x(\mathcal{I})$ denotes the set of $\mathcal{I}$-cluster points of $x$, that is, the set of $\ell \in X$ such that $\{n \in \omega: x_n \in U\} \notin \mathcal{I}$ for all neighborhoods $U$ of $\ell$, cf. 
\cite{MR3883171}).

We continue with some characterizations of density-like ideals, see also \cite[Theorem 4.5]{MR3362232}. 

\begin{prop}\label{equiv:submeasuresdensitylike}
Let $\varphi$ be a lscsm and set $\mathcal{I}:=\mathrm{Exh}(\varphi)$. 
Then the following are equivalent:
\begin{enumerate}[label={\rm (\textsc{d}\arabic{*})}]
\item \label{sub:1} $\textstyle \forall \varepsilon>0, \exists \delta>0, \forall (F_n) \in \mathcal{I}^\omega \cap \mathcal{G}_{\varphi,\delta}, \exists I \in [\omega]^\omega, \varphi\left(\bigcup_{i \in I}F_i\right)<\varepsilon$\textup{;}
\item \label{sub:2} $\textstyle \forall \varepsilon>0, \exists \delta>0, \forall (F_n) \in \mathrm{Fin}^\omega \cap \mathcal{G}_{\varphi,\delta}, \exists I \in [\omega]^\omega, \varphi\left(\bigcup_{i \in I}F_i\right)<\varepsilon$\textup{;}
\item \label{sub:3} $\textstyle \forall \varepsilon>0, \exists \delta>0, \forall (F_n) \in \mathcal{F}_{\mathrm{disj}} \cap \mathcal{G}_{\varphi,\delta}, \exists I \in [\omega]^\omega, \varphi\left(\bigcup_{i \in I}F_i\right)<\varepsilon$ \textup{(}that is, $\mathcal{I}$ is density-like\textup{)}\textup{;}
\item \label{sub:4} $\textstyle \forall \varepsilon>0, \exists \delta>0, \forall (F_n) \in \mathcal{F}_{\mathrm{incr}} \cap \mathcal{G}_{\varphi,\delta}, \exists I \in [\omega]^\omega, \varphi\left(\bigcup_{i \in I}F_i\right)<\varepsilon$\textup{;}
\item \label{sub:5} $\mathcal{I}$ is $((\mathrm{Fin}\times \mathrm{Fin})^+, \emptyset\times \mathrm{Fin})$-calibrated\textup{.}
\end{enumerate}
\end{prop}
\begin{proof}
It is clear that \ref{sub:1} $\implies$ \ref{sub:2} $\implies$ \ref{sub:3} $\implies$ \ref{sub:4}. 

\ref{sub:3} $\Longleftrightarrow$ \ref{sub:5} See \cite[Lemma 6.7]{MR2134228}.


\ref{sub:4} $\implies$ \ref{sub:1} See \cite[Lemma 3.1]{MR2787694}.
\end{proof}



To conclude, every density-like ideal is a generalized density ideal, provided that, in addition, it is $F_\sigma$. 
\begin{prop}\label{thm:casefsigma}
Let $\mathcal{I}$ be an $F_\sigma$ ideal. Then the following are equivalent:
\begin{enumerate}[label={\rm (\textsc{f}\arabic{*})}]
\item \label{item:b0} $\mathcal{I}$ is a generalized density ideal\textup{;}
\item \label{item:b1} $\mathcal{I}$ is a density-like ideal\textup{;}
\item \label{item:b2} $\mathcal{I}$ is $((\mathrm{Fin}\times \mathrm{Fin})^+, \emptyset\times \mathrm{Fin})$-calibrated\textup{;}
\item \label{item:b3} $\mathcal{I}=\mathrm{Fin}$ or $\mathcal{I}=\mathrm{Fin}\oplus \mathcal{P}(\omega)$\textup{.}
\end{enumerate}
\end{prop}
\begin{proof}
\ref{item:b0} $\implies$ \ref{item:b1} This is obvious.

\ref{item:b1} $\implies$ \ref{item:b2} See \cite[Lemma 6.7]{MR2134228}.

\ref{item:b2} $\implies$ \ref{item:b3} See \cite[Proposition 6.8(b)]{MR2134228}.


\ref{item:b3} $\implies$ \ref{item:b0} If $\mathcal{I}=\mathrm{Fin}$ then $\mathcal{I}=\mathrm{Exh}(\varphi_{\bm{\mu}})$, where $\bm{\mu}=(\mu_n)$ and each $\mu_n$ is the Dirac measure on $n \in \omega$. If $\mathcal{I}=\mathrm{Fin}\oplus \mathcal{P}(\omega)$ is represented on $\omega$ as $\{A\subseteq \omega: A \cap 2\omega \in \mathrm{Fin}\}$, then $\mathcal{I}=\mathrm{Exh}(\varphi_{\bm{\mu}})$, where $\mu_n$ is the Dirac measure on $2n$, for each $n \in \omega$.
\end{proof}

\section{Nontall solutions to Question \ref{q:questiondensitylike}}\label{sec:nontall}

In this Section, we provide a positive answer to Question \ref{q:questiondensitylike} by showing that there exists a nontall density-like ideal which is not a generalized density ideal. 

To this aim, given an ideal $\mathcal{I}\subseteq \mathcal{P}(\omega)$, define 
\begin{equation}\label{eq:widehatI}
\widehat{\I}:=h[(\emptyset \times \mathrm{Fin}) \cap (\mathcal{I} \times \emptyset)].
\end{equation}
Note that $(\emptyset \times \mathrm{Fin}) \cap (\mathcal{I} \times \emptyset)$ is an ideal on $\omega^2$, hence $\widehat{\I}$ is an ideal on $\omega$. 
The ideal $\widehat{\I}$ has been introduced and studied by Oliver in \cite[Definition 2.1]{MR2078923}, following an idea of Hjorth. 
It is remarkable that Oliver used these ideals $\widehat{\I}$ to show that, in any model of ZFC, there exist an uncountable family of \emph{Borel} ideals $\mathcal{J}$ such that the quotient Boolean algebras $\mathcal{P}(\omega)/\mathcal{J}$ are pairwise nonisomorphic, see \cite[Theorem 3.2]{MR2078923}, addressing also a well-known question due to Farah \cite{MR1988247}.

\begin{lem}\label{lem:nontall}
Let $\mathcal{I}$ be an ideal. Then $\widehat{\I}$ is a nontall ideal.
\end{lem}
\begin{proof}
It is sufficient to see that $h[\{0\}\times \omega]$ is an infinite set which does not contain any infinite subset in $\mathcal{I}$.
\end{proof} 

Here, we show first that if $\I$ is an analytic P-ideal [density-like, respectively], then so is $\widehat{\I}$; notice that the first claim, with an essentially equivalent proof, can be already found in Oliver's work \cite[Lemma 3.6]{MR2078923}, but we repeat it here for the sake of readers' convenience (e.g., we will make explicit use of the submeasure $\lambda$ defined in \eqref{eq:lambda} also later). 
Then, we prove that $\widehat{\I}$ is not a generalized density ideal whenever $\mathcal{I}$ is tall.


\begin{thm}\label{thm:analyticPdieal}
Let $\mathcal{I}$ be an analytic P-ideal. Then $\widehat{\I}$ is an analytic P-ideal. 
\end{thm}
\begin{proof}
%
Let $\varphi$ be a lscsm such that $\mathcal{I}=\mathrm{Exh}(\varphi)$. 
We may assume that $\mathrm{supp}(\varphi)=\omega$ (indeed, it is easy check that $\mathcal{I}=\mathrm{Exh}(\tilde{\varphi})$, where $\tilde{\varphi}$ is the lscsm defined by $\tilde{\varphi}(A):=\varphi(A)+\sum_{a \in A\setminus \mathrm{supp}(\varphi)}1/(a+1)^{2}$ for all $A\subseteq \omega$). 

Let $\nu$ be the submeasure defined by
\begin{equation}\label{eq:nu}
\forall B\subseteq \omega^2,\quad \nu(B):=\varphi\left(\left\{m \in \omega: B_{(m)} \neq \emptyset\right\}\right).
\end{equation}
To conclude the proof, we claim that $\widehat{\I}=\Exh(\lambda)$, where $\lambda$ is the submeasure defined by 
\begin{equation}\label{eq:lambda}
\forall A\subseteq \omega,\quad \lambda(A):=\nu(h^{-1}[A]). 
\end{equation}
(Note that $\lambda$ is a lscsm and that $\mathrm{supp}(\lambda)=\omega$.)

\medskip

$\mathrm{Exh}(\lambda)\subseteq \widehat{\I}$: Fix $A \in \mathrm{Exh}(\lambda)$ and set $B:=h^{-1}[A]$. Then
\begin{equation}\label{eq:lambdanorm}
\textstyle 0=\|A\|_\lambda=\inf_{F \in \mathrm{Fin}}\lambda(A\setminus F)=\inf_{G \in [\omega^2]^{<\omega}}\nu(B\setminus G).
\end{equation}
First, we want to prove that $B \in \emptyset \times \mathrm{Fin}$. Indeed, in the opposite, there would exist $m \in \omega$ such that $B_{(m)}\notin \mathrm{Fin}$. However, we would obtain
$$
\textstyle \nu(B\setminus G) \ge \nu((\{m\}\times B_{(m)})\setminus G)=\varphi(\{m\})>0
$$
for every finite set $G\subseteq \omega^2$, which contradicts \eqref{eq:lambdanorm}. Secondly, we show that $B \in \mathcal{I}\times \emptyset$. Thanks to \eqref{eq:lambdanorm}, for every $\varepsilon>0$, there exists a finite set $G\subseteq \omega^2$ such that $\nu(B\setminus G)<\varepsilon$. Let $F=\{m\in\omega:G_{(m)} \neq \emptyset\}\in\Fin$. Then 
$$
\varphi\left(\left\{m\in\omega:B_{(m)} \neq \emptyset\right\}\setminus F\right)\leq 
\varphi\left(\left\{m\in\omega:(B\setminus G)_{(m)} \neq \emptyset\right\}\right)=\nu(B\setminus G)<\varepsilon.
$$
By the arbitrariness of $\varepsilon$, we have $B \in \mathcal{I}\times \emptyset$. To sum up, we have $B \in (\emptyset \times \mathrm{Fin}) \cap (\mathcal{I}\times \emptyset)$, so that $h^{-1}[\mathrm{Exh}(\lambda)] \subseteq  (\emptyset\times\Fin)\cap(\I\times\emptyset)$.

\medskip

$\widehat{\I}\subseteq \mathrm{Exh}(\lambda)$: Suppose now that $B \in (\emptyset\times\Fin)\cap(\I\times\emptyset)$ and fix $\varepsilon>0$. Since $B \in \mathcal{I}\times \emptyset$, there exists $F\in\Fin$ such that $\varphi\left(\left\{m\in\omega: B_{(m)} \neq \emptyset\right\}\setminus F\right)<\varepsilon$. However, since $B \in \emptyset \times \mathrm{Fin}$, the set $G:=B \cap (F\times\omega)$ is finite. Hence 
$$
\nu(B\setminus G)=\varphi\left(\left\{m\in\omega:(B\setminus G)_{(m)} \neq \emptyset\right\}\right)=\varphi\left(\left\{m\in\omega:B_{(m)} \neq \emptyset\right\}\setminus F\right)<\varepsilon.
$$


Therefore $\widehat{\I}=\mathrm{Exh}(\lambda)$, which concludes the proof.
\end{proof}

\begin{rmk}\label{rmk:nonpathological}
%
Let us suppose that $\varphi$ is a nonpathological lscsm (in the sense of Farah \cite[Section 1.7]{MR1711328}), that is, 
$$
\textstyle \forall A\subseteq \omega,\quad \varphi(A)=\sup_{\eta \in \mathcal{N}(\varphi)}\eta(A),
$$
where $\mathcal{N}(\varphi)$ stands for the set of finitely additive measures $\eta$ such that $\eta(V)\le \varphi(V)$ for all $V\subseteq \omega$ (note that $\mathcal{N}(\varphi)\neq \emptyset$ as it contains $\eta=0$); 
strictly related notions have been used in game theory, see \cite{MR381720}, and in the context of measure algebras, see \cite{MR2099603, MR701524, MR2456888}.

Then the lscsm $\lambda$ defined in \eqref{eq:lambda} is nonpathological as well. 
%
%
To this aim, fix $A\subseteq \omega$ such that $\lambda(A)\neq 0$ (otherwise the claim is trivial), and recall that 
$$
\lambda(A)=\nu(h^{-1}[A])=\varphi(M), \text{ where }M:=\{m \in \omega: h^{-1}[A]_{(m)} \neq \emptyset\}.
$$ 
In particular, $M$ is nonempty. For each $m \in M$, pick $a_m \in h^{-1}[A]_{(m)}$.  
%
Since $\varphi$ is nonpathological, there exists a sequence $(\eta_n) \in \mathcal{N}(\varphi)^\omega$ such that $\varphi(M)=\lim_n \eta_n(M)$. At this point, define
$$
\textstyle \forall n \in \omega, \forall V \subseteq \omega,\quad \psi_n(V):=\eta_n(\{m \in M: (m,a_m) \in h^{-1}[V]\}).
$$
It is easy to see that each $\psi_n$ is a finitely additive measure. 
Moreover, since each $\eta_n$ is pointwise dominated by $\varphi$, we have
\begin{displaymath}
\begin{split}
\textstyle \forall n \in \omega, \forall V\subseteq \omega,\quad 
\psi_n(V) &\le \varphi(\{m \in M: (m,a_m) \in h^{-1}[V]\})\\
&
\le \varphi(\{m \in M: h^{-1}[V]_{(m)}\neq \emptyset\})\le \lambda(V),
\end{split}
\end{displaymath}
which implies that $(\psi_n) \in \mathcal{N}(\lambda)^\omega$. Lastly, we have that
\begin{displaymath}
\begin{split}
\lambda(A)&\textstyle =\varphi(M)=\lim_{n\to \infty} \eta_n(M)\\
&=\lim_{n \to \infty} \eta_n(\{m \in M: h^{-1}[A]_{(m)} \neq \emptyset\})=\lim_{n\to \infty}\psi_n(A). 
\end{split}
\end{displaymath}
This proves that $\lambda(A)=\sup_{\psi \in \mathcal{N}(\lambda)}\psi(A)$, i.e., $\lambda$ is nonpathological. 

\end{rmk}

Now we show that the submeasure $\lambda$ defined in \eqref{eq:lambda} is density-like provided that $\varphi$ is density-like (for an alternative shorter proof in the case where $\mathcal{I}$ is an Erd\H{o}s-Ulam ideal, see Corollary \ref{cor:simplified} below). 
\begin{thm}\label{thm:lambdadensitylike}
Let $\mathcal{I}$ be a density-like ideal. Then $\widehat{\I}$ is a density-like ideal as well.
\end{thm}
\begin{proof}
Let $\varphi$ be a density-like lscsm such that $\I=\Exh(\varphi)$ and consider the lscsm $\lambda$ defined as in the proof of Theorem \ref{thm:analyticPdieal}. 
Fix $\varepsilon>0$. By Proposition \ref{equiv:submeasuresdensitylike}, there is $\delta=\delta(\varepsilon)>0$ such that for all $(E_n) \in \mathrm{Fin}^\omega \cap \mathcal{G}_{\varphi,\delta}$ there exists $I \in [\omega]^\omega$ with $\varphi\left(\bigcup_{i \in I}E_i\right)<\varepsilon$. 
We claim that the same $\delta$ witnesses the fact that $\lambda$ is density-like.

Fix $(F_n) \in \mathcal{F}_{\mathrm{disj}} \cap \mathcal{G}_{\lambda,\delta}$. 
Define $E_n:=\left\{m \in \omega: h^{-1}[F_n]_{(m)}\neq \emptyset\right\} \in \mathrm{Fin}$ for each $n \in \omega$ and note that 
$$
\varphi(E_n)=\varphi\left(\left\{m \in \omega: h^{-1}[F_n]_{(m)}\neq \emptyset\right\}\right)=\lambda(F_n)<\delta.
$$
Thus, $(E_n) \in \mathrm{Fin}^\omega \cap \mathcal{G}_{\varphi,\delta}$ and there exists $I \in [\omega]^\omega$ with $\varphi\left(\bigcup_{i \in I}E_i\right)<\varepsilon$. Then 
$$
\textstyle \lambda(F)=\varphi(\{m \in \omega: h^{-1}[F]_{(m)} \neq \emptyset\})=\varphi(\bigcup_{i \in I}E_i)<\varepsilon,
$$
where $F:=\bigcup_{i \in I}F_i$. This concludes the proof.
\end{proof}



\begin{thm}\label{thm:notdensityideal}
Let $\mathcal{I}$ be a tall ideal. Then $\widehat{\I}$ is not a generalized density ideal.
\end{thm}
\begin{proof}
Let us suppose that $\widehat{\I}=\mathrm{Exh}(\varphi_{\bm{\mu}})$, where $\bm{\mu}=(\mu_n)$ is a sequence of lscsms such that $(G_n) \in \mathcal{F}_{\mathrm{disj}}$, where $G_n:=\mathrm{supp}(\mu_n)$ for each $n \in \omega$. 

Fix a strictly increasing sequence $(x_n) \in \omega^\omega$ such that 
$$
\forall n \in \omega, \quad x_n \in h[\{n\} \times \omega] 
\,\,\,\text{ and }\,\,\,
|G_n \cap X| \le 1,
$$
where 
$X:=\{x_k: k \in \omega\}$ (it is easy to see that such sequence exists). It follows that $X\notin h[\I\times\emptyset]$, hence $X\notin \widehat{\I}=\mathrm{Exh}(\varphi_{\bm{\mu}})$. 
This implies that 
there exists $\varepsilon>0$ and a strictly increasing sequence $(m_t)\in \omega^\omega$ such that $\mu_{m_t}(X) \ge \varepsilon$ for all $t \in \omega$. 
However, by construction, each $G_{m_t}$ contains at most one element from $X$; hence, exactly one since $\mu_{m_t}(X)\neq 0$, let us say $\{y_t\}:=G_{m_t} \cap X$ for all $t \in \omega$. 
It follows that $\mu_{m_t}(Z) \not\to 0$, where $Z$ stands for any infinite subset of $Y:=\{y_t: t \in \omega\}$, therefore 
$\mathcal{P}(Y) \cap [\omega]^\omega \cap \mathrm{Exh}(\varphi_{\bm{\mu}})=\emptyset$.  
This implies that 
every infinite subset of $Y$ does not belong to $\widehat{\I}$. Considering that $Y \cap h[\{n\}\times \omega]$ is finite for all $n \in \omega$, this contradicts the hypothesis that $\mathcal{I}$ is tall.  
\end{proof}

As an immediate consequence, we obtain the proof of Theorem \ref{thm::main}
\begin{proof}[Proof of Theorem \ref{thm::main}]
Let $\mathcal{I}_d$ be the ideal of density zero sets, which is a tall generalized density ideal. By Lemma \ref{lem:nontall}, Theorem \ref{thm:lambdadensitylike}, and Theorem \ref{thm:notdensityideal}, we get that $\widehat{\I}_d$ is a (nontall) density-like ideal which is not a generalized density ideal.
\end{proof}

At this point, a natural question would be: 
\begin{question}\label{q:questiondensitylikejhsjshjdf}
How many pairwise nonisomorphic ideals $\widehat{\I}$ are there, with $\mathcal{I}$ tall density-like ideal?
\end{question}

To this aim, given ideals $\mathcal{I},\mathcal{J}$ on $\omega$ we say that $\mathcal{I}$ is isomorphic to $\mathcal{J}$ if there exists a bijection $f: \omega\to \omega$ such that $A\in \mathcal{I}$ if and only if $f^{-1}[A] \in \mathcal{J}$ for all $A\subseteq \omega$. 
In addition, we say that $\mathcal{I}$ \emph{is below} $\mathcal{J}$ \emph{in the Kat\v{e}tov order} (written as $\mathcal{I}\le_{\mathrm{K}}\mathcal{J}$) if there exists a function $\kappa:\omega\to \omega$ such that $A \in \mathcal{I}$ implies $\kappa^{-1}[A] \in \mathcal{J}$ for all $A\subseteq \omega$, cf. e.g. \cite{MR3950052}.

Lastly, we recall that an ideal $\mathcal{I}$ is called a \textbf{simple density ideal} if there exists a function $g: \omega \to [0,\infty)$ such that $g(n)\to \infty$, $n/g(n)\not\to 0$ and
$$
\mathcal{I}=\mathcal{Z}_g:=\left\{A\subseteq \omega: \lim_{n\to \infty}\frac{|A\cap n|}{g(n)}=0\right\},
$$
see \cite{MR3391516, MR3771234, MR3950052}. In particular, it has been proved in \cite[Theorem 3.2]{MR3391516} that $\mathcal{Z}_g$ is a density ideal (hence, in particular, a generalized density ideal). It is also evident that $\mathcal{Z}_g$ is tall.

The next result can be deduced also as an immediate consequence of \cite[Theorem 3.4]{MR2078923}; however, the latter one has a different (and seemingly more complicated) proof, hence we present our argument for the sake of completeness.
\begin{thm}\label{thm:countingnonisomorphic}
There are $2^\omega$ tall density-like ideals $\mathcal{I}$ such that the ideals $\widehat{\I}$ are pairwise nonisomorphic. 
\end{thm}
\begin{proof}
Thanks to \cite[Theorem 3]{MR3950052}, there exists a family of simple density ideals $\{\mathcal{I}_\alpha: \alpha<2^\omega\}$ such that $\mathcal{I}_\alpha\not\leq_{\mathrm{K}}\mathcal{I}_\beta$ for all distinct $\alpha,\beta<2^\omega$.

Hence, given distinct $\alpha,\beta<2^{\omega}$, we claim that $\widehat{\I_\alpha}$ is not isomorphic to $\widehat{\I_\beta}$, i.e., there is no bijection $f: \omega^2\to \omega^2$ such that $f[A] \in (\emptyset\times\Fin)\cap(\I_\alpha\times\emptyset)$ if and only if $A \in (\emptyset\times\Fin)\cap(\I_\beta\times\emptyset)$ for all $A\subseteq \omega^2$. 

Suppose that $f: \omega^2 \to \omega^2$ is a bijection and suppose that there exist an infinite set $A\subseteq \omega$ and $k\in\omega$ such that $f[A\times\{0\}]\subseteq\{k\}\times\omega$. Since $\mathcal{I}_\beta$ is tall, there is an infinite $B\subseteq A$ such that $B\in\mathcal{I}_\beta$. Thus, $B\times\{0\}\in(\emptyset\times\Fin)\cap(\I_\beta\times\emptyset)$, but $f[B\times\{0\}]\notin (\emptyset\times\Fin)\cap(\I_\alpha\times\emptyset)$ as $f[B\times\{0\}]\cap(\{k\}\times\omega)$ is infinite. This implies that the function $g:\omega\to \omega$ defined by $f(n,0) \in \{g(n)\} \times \omega$ for all $n \in \omega$ is finite-to-one.

Since $\mathcal{I}_\alpha\not\leq_{\mathrm{K}}\mathcal{I}_\beta$, there exists a (necessarily infinite) set $X\in\mathcal{I}_\alpha$ such that $g^{-1}[X]\notin\mathcal{I}_\beta$. Define $Y:=f[\omega \times \{0\}] \cap(X\times\omega)$. Note that, since $g$ is finite-to-one, then $Y\subseteq f[\omega \times \{0\}] \in \emptyset \times \mathrm{Fin}$. Hence $Y \in (\emptyset \times \mathrm{Fin}) \cap (\mathcal{I}_\alpha \times \emptyset)$.

To conclude the proof, let us suppose for the sake of contradiction that $f^{-1}[Y] \in (\emptyset\times\Fin)\cap(\I_\beta\times\emptyset)$. Hence, in particular, $f^{-1}[Y] \in \mathcal{I}_\beta \times \emptyset$, that is, 
$$
Z:=\{n \in\omega: f^{-1}[Y]_{(n)}\neq \emptyset\} \in \mathcal{I}_\beta.
$$
On the other hand, we have 
\begin{displaymath}
\begin{split}
Z&=\{n \in \omega: \exists k \in \omega, (n,k) \in f^{-1}[Y]\}=\{n \in \omega: \exists k \in \omega, f(n,k) \in Y\}\\
&=\{n \in \omega: f(n,0) \in Y\}\supseteq \{n \in \omega: g(n) \in X\}=g^{-1}[X] \notin \mathcal{I}_\beta,
\end{split}
\end{displaymath}
where we used that, if $g(n) \in X$, then $f(n,0) \in \{g(n)\}\times \omega$ and $f(n,0) \in f[\omega \times \{0\}]$, i.e., $f(n,0) \in Y$. This completes the proof.
\end{proof}

Thus, we can answer Question \ref{q:questiondensitylikejhsjshjdf}:
\begin{thm}\label{cor:2omega}
There are $2^\omega$ nonpathological and pairwise nonisomorphic nontall density-like ideals which are not generalized density ideals.
\end{thm}
\begin{proof}
Let $\mathcal{I}$ be a simple density ideal. Then $\mathcal{I}$ is a density ideal and, in particular, it is nonpathological. Thanks to Remark \ref{rmk:nonpathological}, $\widehat{\mathcal{I}}$ is nonpathological as well. The claim follows by Lemma \ref{lem:nontall}, Theorem \ref{thm:lambdadensitylike}, Theorem \ref{thm:notdensityideal}, and Theorem \ref{thm:countingnonisomorphic}. 
\end{proof}

\section{Tall solutions to Question \ref{q:questiondensitylike}}\label{sec:tall}

In the previous Section we have shown that there exists a nontall density-like ideal which is not a generalized density ideal, providing a positive answer to Question \ref{q:questiondensitylike}. Hence, we may ask:
\begin{question}\label{sec:tallquestionoriiginal}
Does there exist a \emph{tall} density-like ideal which is not a generalized density ideal?
\end{question}
In this Section, we answer positively also Question \ref{sec:tallquestionoriiginal}. 



\begin{defi}
\label{equi-density-like}
A sequence $\bm{\mu}=(\mu_n)$ of lscsms is \textbf{equi-density-like} if for all $\varepsilon>0$ there exists $\delta >0$ such that for all $n\in\omega$ and $(F_k) \in \mathcal{F}_{\mathrm{disj}} \cap \mathcal{G}_{\mu_n,\delta}$ there exists an infinite set $I\subseteq \omega$ such that $\mu_n(\bigcup_{i \in I}F_i) < \varepsilon$. 
\end{defi}
In words, each lscsm $\mu_n$ is density-like and the choice of $\delta=\delta(\varepsilon)$ is uniform within all $\mu_n$s.

\begin{thm}\label{sup}
Let $\bm{\mu}=(\mu_n)$ be a sequence of lscsms with pairwise disjoint supports. Then $\varphi:=\sup_n\mu_n$ is density-like if and only if $\bm{\mu}$ is equi-density-like.
\end{thm}
\begin{proof}
Define $S_n:=\mathrm{supp}(\mu_n)$ for each $n \in \omega$.

\medskip

\textsc{Only If part.} Suppose that $\bm{\mu}$ is not equi-density-like. Then there exists $\varepsilon>0$ such that for all $\delta>0$ there are $n \in \omega$ and $(F_k) \in \mathcal{F}_{\mathrm{disj}} \cap \mathcal{G}_{\mu_n,\delta}$ for which $\mu_n(\bigcup_{i \in I}F_i) \ge \varepsilon$ whenever $I\in [\omega]^{\omega}$. 
We claim that this $\varepsilon>0$ witnesses that $\varphi$ is not density-like. 
To this aim, fix any $\delta>0$ and let $n$ and $(F_k)$ be as before. Define $E_k:=F_k \cap S_n$ for all $k \in \omega$. Then $\varphi(E_k)=\mu_n(E_k)<\delta$ and $(E_k) \in \mathcal{F}_{\mathrm{disj}} \cap \mathcal{G}_{\varphi, \delta}$. At the same time, we have
$$
\textstyle \forall I \in [\omega]^{\omega}, \quad \varphi\left(\bigcup_{i \in I}E_i\right)=\mu_n\left(\bigcup_{i \in I}E_i\right)=\mu_n\left(\bigcup_{i \in I}F_i\right)\ge \varepsilon.
$$
Therefore $\varphi$ is not density-like.

\medskip

\textsc{If part.} Conversely, suppose that $\bm{\mu}$ is equi-density-like, and fix $\varepsilon>0$. Then there exists a sufficiently small $\delta \in (0,\nicefrac{\varepsilon}{4})$ such that 
\begin{equation}\label{eq:hypotesisequidensitylike}
\textstyle \forall n \in \omega, \forall (F_k) \in \mathcal{F}_{\mathrm{disj}} \cap \mathcal{G}_{\mu_n,\delta}, \exists I \in [\omega]^\omega, \mu_n\left(\bigcup_{i \in I}F_i\right)<\frac{\varepsilon}{4}.
\end{equation}
Fix $(F_k) \in \mathcal{F}_{\mathrm{disj}} \cap \mathcal{G}_{\varphi,\delta}$ and define $(T_k) \in \mathrm{Fin}^\omega$ by $T_k:=\{n \in \omega: F_k \cap S_n \neq \emptyset\}$ for all $k \in \omega$. At this point, we claim that there exist a sequence of infinite sets $(X_j) \in (\mathrm{Fin}^+)^\omega$ and an increasing sequence $(i_j) \in \omega^\omega$ such that, for all $j \in \omega$:
\begin{enumerate}[label=\textup{(}\roman*\textup{)}]
\item \label{item1} 
$i_j=\min X_j$,
\item \label{item2} 
$X_{j+1}\subseteq X_j\setminus\{i_j\}$, and 
\item \label{item3} 
$\mu_t\left(F^{(j)} \, \cup \, \bigcup_{i\in X_{j+1}} F_i\right)<\frac{\varepsilon}{2}$ for each $t\in T^{(j)}$, where $F^{(j)}:=\bigcup_{k\le j} F_{i_k}$ and $T^{(j)}:=\bigcup_{k\le j} T_{i_k}$.
\end{enumerate}

%


We define these sequences recursively. 
Start with $X_0=\omega$ and $i_0=0$. 
Suppose now that $X_k$ and $i_k$ have been defined for all $k\le j \in \omega$ and satisfy \ref{item1}-\ref{item3}. 
If $T:=T_{i_j}\setminus T^{(j-1)}$ is empty, set $X_{j+1}:=X_j\setminus\{i_j\}$ and $i_{j+1}:=\min X_{j+1}$; 
if $T\neq\emptyset$, since $\mu_n(F_k)\leq \varphi(F_k)<\delta$ for all $n,k\in\omega$, we can find $X_{j+1}\subseteq X_j\setminus\{i_j\}$ such that $\mu_t\left(\bigcup_{i \in X_{j+1}}F_i\right)<\frac{\varepsilon}{4}$ for all $t \in T$; 
finally, set $i_{j+1}:=\min X_{j+1}$. 
If $j\neq 0$ and $t \in T^{(j-1)}$, it follows by the induction hypothesis that
$$
\textstyle \mu_t\left(F^{(j)} \, \cup \, \bigcup_{i\in X_{j+1}} F_i\right)\le \mu_t\left(F^{(j-1)} \, \cup \, \bigcup_{i\in X_{j}} F_i\right) < \frac{\varepsilon}{2}; 
$$
On the other hand, if $j=0$ or $t \in T$, then
$$
\textstyle \mu_t\left(F^{(j)} \, \cup \, \bigcup_{i\in X_{j+1}} F_i\right)\le \mu_t(F_{i_j})+\mu_t\left(\bigcup_{i\in X_{j+1}} F_i\right) < \delta+\frac{\varepsilon}{4}< \frac{\varepsilon}{2},
$$
which proves the condition \ref{item3} and completes the induction.

To complete the proof, note that if $n \notin \bigcup_k T^{(k)}$ then $\mu_n\left(\bigcup_k F^{(k)}\right)=0$ for all $n \in \omega$. Moreover, if $n \in \bigcup_k T^{(k)}$ then 
$$
\textstyle \forall j \in \omega, \quad \mu_n\left(\bigcup_k F^{(k)}\right)\le \mu_n\left(F^{(j)} \, \cup \, \bigcup_{i\in X_{j+1}} F_i\right)<\frac{\varepsilon}{2}.
$$
Thus $\varphi\left(\bigcup_k F^{(k)}\right)\le \frac{\varepsilon}{2}<\varepsilon$, which shows that $\varphi$ is density-like.
\end{proof}

It is worth noting that the above proof works also if $\bm{\mu}$ is a sequence of lscsms such that $\{n\in\omega:k\in\textrm{supp}(\mu_n)\}$ is finite for all $k\in\omega$, in the same spirit of \cite[Proposition 5.4]{MR3436368}, cf. Proposition \ref{prop:supportgeneralizeddensityideal}.

The aim of the next example is twofold: 
first of all, it shows that there exist sequences of lscsms that are not equi-density-like (hence, their pointwise supremum is not density-like); 
secondly, it proves that the ideal $\mathrm{Exh}(\sup_n \mu_n)$ depends on the sequence of lscsms $(\mu_n)$, not on the sequence of ideals $(\mathrm{Exh}(\mu_n))$, that is, there are two sequences of lscsms $(\mu_n)$ and $(\nu_n)$ such that $\mathrm{Exh}(\mu_n)=\mathrm{Exh}(\nu_n)$ for each $n\in \omega$ and, on the other hand, $\mathrm{Exh}(\sup_n \mu_n)\neq \mathrm{Exh}(\sup_n \nu_n)$.

\begin{example}
Let $(I_{n,m})_{n,m\in\omega}$ be a sequence of pairwise disjoint finite subsets of $\omega$ such that $|I_{n,m}|=2^m$ for each $n,m\in\omega$. 
Moreover, let $(\eta_{n,m})_{n,m\in\omega}$ be the sequence of probability measures on $\omega$ defined by
$$
\forall n,m \in \omega, \forall A\subseteq \omega,\quad \eta_{n,m}(A)=\frac{|A\cap I_{n,m}|}{2^m}.
$$ 
Then, define the sequences of lscsms $\bm{\mu}=(\mu_n)$ and $\bm{\nu}=(\nu_n)$ by
$$
\forall n \in \omega,\quad 
\mu_n=\sup_{m \in \omega}\eta_{n,m}
\,\,\,\text{ and }\,\,\,
\nu_n=\sup_{M \in [\omega]^{n+1}}\sum_{m \in M} \eta_{n,m},
$$
where the suprema are meant in the pointwise order. 

On the one hand, it is easy to see that $\|A\|_{\mu_n}=0$ if and only if $\|A\|_{\nu_n}=0$, so that $\mathrm{Exh}(\mu_n)$ and $\mathrm{Exh}(\nu_n)$ are density ideals and they coincide for each $n \in \omega$. 

On the other hand, $\mathrm{Exh}(\sup_n \mu_n)=\mathrm{Exh}(\sup_{n,m}\eta_{n,m})$ is a density ideal (hence, in particular, it is a density-like ideal), and it is \emph{not} equal to $\mathrm{Exh}(\sup_n \nu_n)$. Indeed, we will prove that $\mathrm{Exh}(\sup_n \nu_n)$ is not density-like. Thanks to Theorem \ref{sup}, this is equivalent to show that $\bm{\nu}$ is not equi-density-like. 

To this aim, put $\varepsilon=1$ and fix any $\delta>0$. There is $k\in\omega$ such that $\frac{1}{2^k}<\delta$. We will find a sequence $(F_n) \in \mathcal{F}_{\mathrm{disj}} \cap \mathcal{G}_{\nu_{2^k},\delta}$ such that $\nu_{2^k}(\bigcup_{i \in I}F_i) \geq \varepsilon$ for each infinite $I\subseteq\omega$. 
For each $n \in \omega$, fix a subset $F_n\subseteq I_{2^k,k+n}$ such that $|F_n|=2^n$. Note that 
$\nu_{2^k}(F_n)=\eta_{2^k,n+k}(F_n)=1/2^{k}<\delta$ 
for all $n \in \omega$. Therefore $(F_n) \in \mathcal{F}_{\mathrm{disj}} \cap \mathcal{G}_{\nu_{2^k},\delta}$. Lastly, fix an infinite set $I \subseteq \omega$ and a subset $M\subseteq I$ such that $|M|=2^k+1$. Then
$$
\nu_{2^k}(\bigcup_{i \in I}F_i) \geq 
\nu_{2^k}(\bigcup_{i \in M}F_{i}) \ge 
\sum_{m \in M}\eta_{2^k,k+m}(\bigcup_{i \in M}F_{i}) 
=\sum_{m \in M}\eta_{2^k,k+m}(F_m) >1,
$$
which proves that $\bm{\nu}$ is not equi-density-like.
\end{example}

\begin{defi}
A lscsm $\varphi$ is \textbf{strongly-density-like} if there is a constant $c=c(\varphi)>0$ such that, for all $\varepsilon>0$ and $(F_k) \in \mathcal{F}_{\mathrm{disj}} \cap \mathcal{G}_{\varphi,c \varepsilon}$, there exists an infinite set $I \subseteq \omega$ such that $\varphi(\bigcup_{i \in I}F_i) < \varepsilon$.
\end{defi}

\begin{rmk}
\label{rem-strongly-density-like}
Observe that if $\bm{\mu}$ is a sequence of lscsms with pairwise disjoint finite supports, then $\varphi_{\bm{\mu}}=\sup_n \mu_n$ is strongly-density-like with any constant $c(\varphi_{\bm{\mu}})<1$. Thus, if $\I$ is a generalized density ideal, then there is a strongly-density-like lscsm $\varphi$ with $\I=\Exh(\varphi)$. 
\end{rmk}

In the following example we show that there exist density-like lscsms which are not strongly-density-like, cf. also Section \ref{sec:concludingrmk}.

\begin{example}
\label{ex-strongly-density-like}
Let $h: \omega^2 \to \omega$ be the bijection defined in \eqref{eq:bijectionh}. For each $n \in \omega$ define $X_n:=h^{-1}[\{n\}\times \omega]$, so that $\{X_n\}$ is a partition of $\omega$ into infinite sets. Define $a_k:=1/(k+2)!$ for each $k \in \omega$ and note that $a_k\to 0$ as $k\to+\infty$ and $a_{k-1}>(k+1)a_k$ for all $k>0$. Moreover, set $\varphi:=\sup_n \mu_n$, where $(\mu_n)$ is the sequence of lscsms given by
$$
\forall n \in \omega, \forall A\subseteq \omega,\quad \mu_n(A)=a_n \min\{n+1,|A\cap X_n|\}.
$$
We claim that $\varphi$ is a density-like lscsm which is not strongly-density-like. 

First, we show that $\varphi$ is not strongly-density-like. To this aim, fix an arbitrary constant $c>0$ and a positive integer $k$ such that $\nicefrac{1}{k}\le c$. Then, set $\varepsilon:=(k+1)a_k$ and $F_n:=\{h^{-1}(k,n)\}$ for each $n \in \omega$. It follows that 
$$
\textstyle \forall n\in \omega,\quad \varphi(F_n)=\mu_k(F_k)=a_k< \frac{(k+1)a_k}{k}\le c\varepsilon,
$$
hence $(F_n) \in \mathcal{F}_{\mathrm{disj}} \cap \mathcal{G}_{\varphi,c\varepsilon}$. On the other hand, for each infinite set $I\subseteq \omega$, we have 
$
\varphi(\bigcup_{i \in I}F_i)= \mu_k(\bigcup_{i \in I}F_i)=(k+1)a_k=\varepsilon.
$ 

Now let us show that $\varphi$ is density-like. Fix $\varepsilon>0$ and define $\delta:=a_k$, where $k$ is an integer such that $(k+1)a_k<\varepsilon$. Fix also $(F_k) \in \mathcal{F}_{\mathrm{disj}} \cap \mathcal{G}_{\varphi,\delta}$. Note that $F_k \cap \bigcup_{n\le k-1}X_n=\emptyset$ for each $k \in \omega$: indeed 
$$
\textstyle \forall x \in \bigcup_{n\le k-1}X_n,\quad \varphi(\{x\}) \ge a_{k-1} > (k+1)a_k \ge a_k =\delta.
$$ 
Therefore
$$
\forall I \in [\omega]^\omega,\quad \varphi(\bigcup_{i \in I}F_i)\le \varphi(\bigcup_{n\ge k}X_n)=\sup_{k\ge n} \mu_k(X_k)=(k+1)a_k<\varepsilon,
$$
concluding the proof.
\end{example}

\begin{prop}\label{lem:DL}
Fix $q \in [1,\infty)$. Then the set of strongly-density-like lscsms is $q$-convex, that is, for each strongly-density-like lscsms $\varphi_1,\ldots,\varphi_k$ and $a_1,\ldots,a_k \in [0,1]$ with $\sum_{i\le k} a_i=1$, the lscsm 
$
\varphi:=(\sum_{i\le k}a_i \varphi_i^q)^{1/q}
$  
is strongly-density-like. In addition, a witnessing constant of $\varphi$ is $c(\varphi)=\frac{1}{2}\min\{c(\varphi_1),\ldots,c(\varphi_k)\}$. 
\end{prop}
\begin{proof}
Let $\varphi_1,\ldots,\varphi_k$ be strongly-density-like lscsms, fix $a_1,\ldots,a_k \in [0,1]$ with $\sum_{i\le k} a_i=1$, and define the lscsm 
$\varphi:=(\sum_{i\le k}a_i \varphi_i^q)^{1/q}$. Set $c:=\frac{1}{2}\min_{i\le k}c(\varphi_i)$, so that $2c$ is a witnessing constant for each $\varphi_i$. 

Fix $\varepsilon>0$ and a sequence $(F_j) \in \mathcal{F}_{\mathrm{disj}}\cap \mathcal{G}_{\varphi,c\varepsilon}$. For each $j \in \omega$ and $i \in \{1,\ldots,k\}$ define the integer 
$
z_{i,j}:=\left\lfloor \varphi_i(F_j)/c\varepsilon\right\rfloor.
$ 
Note that $z_{i,j}<a_i^{-1/q}$, indeed
$$
\textstyle \forall i=1,\ldots,k,\quad c\varepsilon>\varphi(F_j)\ge a_i^{1/q}\varphi_i(F_j) \ge a_i^{1/q}z_{i,j}c\varepsilon.
$$
Considering that $L:=\prod_{i=1}^k (\omega \cap [0,a_i^{-1/q}))$ is finite, there exists $\ell=(\ell_1,\ldots,\ell_k) \in L$ and an infinite set $J\subseteq \omega$ such that $z_{i,j}=\ell_i$ for all $i =1,\ldots,k$ and $j \in J$. 

At this point, observe that $\sum_{i\le k}a_i\ell_i^q<1$, indeed 
$$
\sum_{i\le k}a_i(c\varepsilon\ell_i)^q=\sum_{i\le k}a_i(c\varepsilon z_{i,\min J})^q\le \sum_{i\le k}a_i\varphi_i(F_{\min J})^q=\varphi(F_{\min J})^q < (c\varepsilon)^q.
$$
Since each $\varphi_i$ is strongly-density-like with witnessing constant $2c$, we have that $\varphi_i(F_j)<(z_{i,j}+1)c\varepsilon=(2c)\cdot \frac{(\ell_i+1)\varepsilon}{2}$ for all $j \in J$. Hence there exists an infinite set $T\subseteq J$ such that 
$
\varphi_i(\bigcup_{t \in T}F_t)<\frac{(\ell_i+1)\,\varepsilon}{2}.
$ 
for all $i=1,\ldots,k$. 

Thanks to Minkowski's inequality, we conclude that
$$
 \textstyle 
\varphi(\bigcup_{t \in T}F_t)<\frac{\varepsilon}{2}\left(\sum_{i\le k}a_i (\ell_i+1)^q\right)^{1/q}\le \frac{\varepsilon}{2}\left(\left(\sum_{i\le k}a_i\ell_i^q\right)^{1/q}+1\right)<\varepsilon,
$$
which proves that $\varphi$ is strongly-density-like.
\end{proof}

\begin{defi}
\label{def:DL}
%
%
An ideal $\mathcal{I}$ is said to be a \textbf{DL-ideal} if it is isomorphic to some $\mathrm{Exh}(\psi)$, where $\psi=\psi(\bm{\varphi},q,a,S)$ is a lscsm on $\omega^2$ defined by
\begin{equation}\label{eq:definitiondlideal}
\forall A\subseteq \omega^2,\quad \psi(A):=\sup_{n \in \omega}\,(\sum_{k \in S_n}a_k \varphi_k^{q_n}(A_{(k)}))^{1/q_n},
\end{equation}
where $\bm{\varphi}=(\varphi_n)$ is a sequence a strongly-density-like lscsms on $\omega$, $q=(q_n)$ is a sequence in $[1,\infty)$, and $a=(a_n) \in [0,1]^\omega$ and $S=(S_n) \in \mathcal{F}_{\mathrm{disj}}$ are sequences such that $\sum_{k \in S_n}a_k=1$ for all $n \in \omega$.
\end{defi}

For the rest of this section, it may be helpful to take in mind the following example of DL-ideals.
\begin{example}\label{example:answerQuestion3}
Fix a strictly increasing sequence $(\iota_n)$ of nonnegative integers such that $\iota_0:=0$ and $\lim_n |I_n|=\infty$, where $I_n:=[\iota_n,\iota_{n+1})$ for all $n \in \omega$. 
In particular, thanks to Remark \ref{rem-strongly-density-like}, the map $\mathcal{P}(\omega) \to [0,\infty]$ defined by $S\mapsto \sup_m |S\cap I_m|/|I_m|$ is a strongly-density-like lscsm. It follows that $\mathcal{I}:=\{h[A]: A \in \mathrm{Exh}(\psi)\}$ is a DL-ideal, where $\psi: \omega^2 \to [0,\infty]$ is the lscsm defined by
$$
\forall A\subseteq \omega^2, \quad 
\psi(A):=\sup_{n \in \omega} \left(\frac{1}{|I_n|}\sum_{k \in I_n}\sup_{m\in \omega} \frac{|A_{(k)}\cap I_m|}{|I_m|}\right).
$$

As it will be shown below in Corollary \ref{cor:finalcortalldensitylike}, the above ideal provides an affirmative answer to Question \ref{sec:tallquestionoriiginal}.
\end{example}

\begin{rmk}\label{rmk:nonpathologicaltall}
It is not difficult to see that, if $\psi=\psi(\bm{\varphi},q,a,S)$ is a lscsm on $\omega^2$ as in \eqref{eq:definitiondlideal} such that $q$ is the constant sequence $(1)$ and each $\varphi_n$ is nonpathological, then $\psi$ is nonpathological as well. 
For each lscsm $\phi$ on $\omega$ and for each $k \in \omega$, let $\tilde{\phi}^{(k)}$ be the lscsm on $\omega^2$ defined by $\tilde{\phi}^{(k)}(A):=\phi(A_{(k)})$ for all $A\subseteq \omega^2$. Then, with the same notation of Remark \ref{rmk:nonpathological}, we know that
$$
\forall A\subseteq \omega^2, \forall k \in \omega,\quad \varphi_k(A_{(k)})=\sup_{\eta \in \mathcal{N}(\varphi_k)}\eta(A_{(k)})
=\sup_{\tilde{\eta}^{(k)} \in \mathcal{N}_k(\varphi_k)}\tilde{\eta}^{(k)}(A),
$$
where $\mathcal{N}_k(\varphi_k):=\{\tilde{\eta}^{(k)}: \eta \in \mathcal{N}(\varphi_k)\}$ (note that, if $k\neq k^\prime$, then measures in $\mathcal{N}_k(\varphi_k)$ have disjoint supports from measures in $\mathcal{N}_{k^\prime}(\varphi_{k^\prime})$). 
Let $\tilde{\mathcal{N}}(\psi)$ be the set of finitely additive measures on $\omega^2$ which are pointwise dominated by $\psi$. 
Then
\begin{displaymath}
\begin{split}
\forall A\subseteq \omega^2,\quad \sup_{\tilde{\eta} \in \tilde{\mathcal{N}}(\psi)}\tilde{\eta}(A)&\le \psi(A)=\sup_{n \in \omega} \sum_{k \in S_n} a_k \sup_{\tilde{\eta}^{(k)} \in \mathcal{N}_k(\varphi_k)} \tilde{\eta}^{(k)}(A)\\
&=\sup_{n \in \omega} \sup_{\substack{\tilde{\eta}^{(k)} \in \mathcal{N}_k(\varphi_k), \\ \text{ with }k \in S_n}} \sum_{k \in S_n} a_k \tilde{\eta}^{(k)}(A) \le \sup_{\tilde{\eta} \in \tilde{\mathcal{N}}(\psi)}\tilde{\eta}(A),
\end{split}
\end{displaymath}
where the last inequality is justified by the fact that each $\sum_{k \in S_n} a_k \tilde{\eta}^{(k)}$ is finitely additive measure which is dominated by $\psi$. Therefore $\psi$ is nonpathological. 
\end{rmk}

\begin{rmk}
Each generalized density ideal is a $\mathrm{DL}$-ideal. Indeed, let $\bm{\mu}=(\mu_n)$ be a sequence of lscsms with pairwise disjoint finite supports such that $\mathcal{I}=\mathrm{Exh}(\sup_n \mu_n)$. Moreover, let $h: \omega^2 \to \omega$ be the bijection defined in \eqref{eq:bijectionh} and $f: \omega^2 \to \omega$ be a bijection such that 
$$
\forall n,m \in \omega,\quad \mathrm{supp}(\mu_{h(n,m)})\subseteq f[\{n\}\times\omega].
$$
Then it suffices to set $\bm{\varphi}=(\varphi_n)$, where $\varphi_n(A)=\sup_{m\in\omega}\mu_{h(n,m)}(f[\{n\}\times A])$ for all $n\in\omega$ and $A\subseteq\omega$, $q=(1)$, $a=(1)$, and $S=(\{n\})$. It follows that each $\varphi_n$ is strongly-density-like (cf. Remark \ref{rem-strongly-density-like}) and that $\mathcal{I}$ is isomorphic to $\Exh(\psi)$, where $\psi=\psi(\bm{\varphi},q,a,S)$ is the lscsm on $\omega^2$ defined as in \eqref{eq:definitiondlideal}.
\end{rmk}

\begin{prop}\label{prop:DL}
Let $\I$ be an Erd\H{o}s–Ulam ideal. Then $\widehat{\I}$ is a $\mathrm{DL}$-ideal.
\end{prop}
\begin{proof}
By \cite[Example 1.2.3.(d), Theorem 1.13.3(a), and Lemma 1.13.9.(Z3)]{MR1711328}, there is a sequence $\bm{\mu}=(\mu_n)$ of probability measures on $\omega$ such that $\mathcal{I}=\mathrm{Exh}(\sup_n \mu_n)$, and $(M_n) \in \mathcal{F}_{\mathrm{disj}}$, where $M_n:=\mathrm{supp}(\mu_n)$ for each $n \in \omega$. Then $\widehat{\I}$ is isomorphic to the ideal $\mathrm{Exh}(\nu)$ on $\omega^2$, where $\nu$ is the lscsm defined by 
$$
\textstyle \forall A\subseteq \omega^2,\quad \nu(A)=\sup_{n\in \omega} \mu_n(\{k \in \omega: A_{(k)}\neq \emptyset\}),
$$
cf. \eqref{eq:nu} in the proof of Theorem \ref{thm:analyticPdieal}. 

To conclude the proof, we show that $\nu=\psi$, for some $\psi=\psi(\bm{\varphi},q,a,S)$. To this aim, let $\bm{\varphi}$ be the constant sequence $(\varphi)$, where $\varphi$ is the strongly-density-like lscsm defined by $\varphi(\emptyset)=0$ and $\varphi(S)=1$ for all nonempty $S\subseteq \omega$. 
Let also $q$ be the constant sequence $(1)$, $S_n=M_n$, and $a_k:=\sup_n \mu_n(\{k\})$ for all $n \in \omega$ (note that $\sum_{k\in S_n}a_k=1$ for all $n \in \omega$). It follows that
\begin{displaymath}
\forall A\subseteq \omega^2,\quad \psi(A)=\sup_{n \in \omega}\sum_{k \in S_n}a_k \varphi(A_{(k)})=\sup_{n \in \omega}\mu_n(\{k \in S_n: A_{(k)}\neq \emptyset\})=\nu(A).
\end{displaymath}
Therefore $\widehat{\I}$ is a $\mathrm{DL}$-ideal.
\end{proof}

\begin{thm}\label{DL-density-like}
Let $\psi=\psi(\bm{\varphi},q,a,S)$ be a lscsm on $\omega^2$ 
as in \eqref{eq:definitiondlideal} 
such that 
$\inf_n c(\varphi_n)>0$. Then $\psi$ is a density-like lscsm.
\end{thm}
\begin{proof}
For each $k \in \omega$ define the lscsm $\tilde{\varphi}_k$ on $\omega^2$ by
$$
\forall A\subseteq \omega^2,\quad 
\tilde{\varphi}_k(A)=\varphi_k(A_{(k)}).
$$
Then each $\tilde{\varphi}_k$ is a strongly-density-like lscsm such that $c(\tilde{\varphi}_k)=c(\varphi_k)$. Moreover, for each $n \in \omega$, define the lscsm $\psi_n$ by
$$
\forall A\subseteq \omega^2,\quad \psi_n(A)=(\sum_{k \in S_n}a_k\tilde{\varphi}_k(A)^{q_n})^{1/q_n}.
$$
It follows by Proposition \ref{lem:DL} that each lscsm $\psi_n$ is strongly-density-like with witnessing constant $c(\psi_n)=\frac{1}{2}\min\{c(\tilde{\varphi}_k): k \in S_n\}$. Since $\inf_n c(\varphi_n)>0$, we have also $\inf_n c(\psi_n)>0$, which implies that $(\psi_n)$ is equi-density-like sequence of lscsms with pairwise disjoint supports (indeed $\mathrm{supp}(\psi_n)=\bigcup_{k \in S_n}\{k\}\times \mathrm{supp}(\varphi_n)\subseteq S_n \times \omega$). Therefore, thanks to Theorem \ref{sup}, $\psi=\sup_n \psi_n$ is density-like.
\end{proof}

The following corollary is immediate.
\begin{cor}\label{cor:samevarphinimpliesdensitylike}
Let $\mathcal{I}$ be a $\mathrm{DL}$-ideal isomorphic to $\mathrm{Exh}(\psi)$, where $\psi=\psi(\bm{\varphi},q,a,S)$ is a lscsm 
on $\omega^2$ 
such that $\bm{\varphi}$ is a constant sequence. 
Then $\mathcal{I}$ is density-like.
\end{cor}

Note that, as it follows from the proof of Proposition \ref{prop:DL}, if $\mathcal{I}$ is an Erd\H{o}s–Ulam ideal, then $\widehat{\mathcal{I}}$ is isomorphic to $\mathrm{Exh}(\psi)$, where $\psi=\psi(\bm{\varphi},q,a,S)$ is a lscsm such that $\bm{\varphi}$ is a constant sequence. 
Therefore, thanks to Corollary \ref{cor:samevarphinimpliesdensitylike}, we obtain: 
\begin{cor}\label{cor:simplified}
Let $\mathcal{I}$ be an Erd\H{o}s–Ulam ideal. Then $\widehat{\mathcal{I}}$ is a density-like ideal.
\end{cor}

Note that Corollary \ref{cor:simplified} is also a consequence of Theorem \ref{thm:lambdadensitylike}; however, the proof of the latter uses Proposition \ref{equiv:submeasuresdensitylike}, which in turn relies on \cite[Lemma 3.1]{MR2787694}.

Note that, thanks to Proposition \ref{lem:DL}, $\psi$ is the pointwise supremum of strongly-density-like lscsms. Now we show that, under some additional hypotheses, the exhaustive ideal generated by $\psi$ is tall.

\begin{prop}\label{proposition:DLidealtall}
Let $\psi=\psi(\bm{\varphi},q,a,S)$ be a lscsm on $\omega^2$ as in \eqref{eq:definitiondlideal} and assume that $\mathrm{Exh}(\varphi_n)$ is tall for all $n \in \omega$, $M:=\sup_{n,k}\varphi_n(\{k\})<\infty$, and 
$\max\{a_k^{1/q_n}: k \in S_n\}\to 0$ as $n\to \infty$. 
Then $\mathrm{Exh}(\psi)$ is tall.
\end{prop}
\begin{proof}
Let $A\subseteq \omega^2$ be an infinite set. If $A \in \mathrm{Exh}(\psi)$ then the claim is trivial. Hence, suppose hereafter that $A\in \mathrm{Exh}(\psi)^+$, that is, 
\begin{equation}\label{eq:psiApositive}
\|A\|_\psi=\inf_{F \in [\omega^2]^{<\omega}}\,\sup_{n \in \omega}\,(\sum_{k \in S_n}a_k \varphi_k^{q_n}(A_{(k)}\setminus F))^{1/q_n}>0.
\end{equation}

Now, suppose that there exists $m \in \omega$ such that $A_{(m)}$ is infinite. Since $\mathrm{Exh}(\varphi_m)$ is tall, there exists an infinite set $B\subseteq A\cap (\{m\}\times \omega)$ such that $B_{(m)} \in \mathrm{Exh}(\varphi_m)$. It follows 
by the definition of $\psi$ 
that $B \in \mathrm{Exh}(\psi)$. 

Otherwise $A \in (\emptyset \times \mathrm{Fin}) \cap \mathrm{Exh}(\psi)^+$, so that $A_{(m)}$ is finite for each $m \in \omega$. Let $B$ be an infinite subset of $A$ such that $|B \cap \bigcup_{m \in S_n}A_{(m)}|\le 1$ for all $n \in \omega$ (which exists, otherwise $A$ itself would be finite, contradicting \eqref{eq:psiApositive}). It follows that
\begin{displaymath}
\begin{split}
\|B\|_\psi&\le \inf_{m\in \omega} \|B\setminus (\bigcup_{i\le m}S_i \times \omega)\|_{\psi} \le \inf_{m \in \omega}\psi( B\setminus (\bigcup_{i\le m}S_i \times \omega))\\
&\le \inf_{m \in \omega} \sup_{n\ge m}\,(\sum_{k \in S_n}a_k \varphi_k^{q_n}(B_{(k)}))^{1/q_n} \\
&\le M \, \limsup_{n\to\infty} \max\{a_k^{1/q_n}: k \in S_n\} =0.
\end{split}
\end{displaymath}
Therefore $B \in \mathrm{Exh}(\psi)$, concluding the proof.
\end{proof}

As a consequence, we obtain that:
\begin{cor}\label{cor:talldlideal}
Let $\psi=\psi(\bm{\varphi},q,a,S)$ be a lscsm on $\omega^2$ as in \eqref{eq:definitiondlideal} and assume that $\bm{\varphi}$ is the constant sequence $(\varphi)$, $q$ is a bounded sequence, $\mathrm{Exh}(\varphi)$ is tall, and $\lim_n a_n=0$. Then $\mathrm{Exh}(\psi)$ is tall.
\end{cor}
\begin{proof}
First of all, we have $\sup_n \varphi(\{n\})<\infty$: indeed, in the opposite case, there would exist an increasing sequence $(n_k)$ in $\omega$ such that $\varphi(\{n_k\}) \ge k$ for all $k$ and every infinite subset of $\{n_k: k \in \omega\}$ would not belong to $\mathrm{Exh}(\varphi)$, contradicting the hypothesis that $\mathrm{Exh}(\varphi)$ is tall. Moreover, since $Q:=\sup_n q_n <\infty$, we obtain
$$
\textstyle \lim_n \max \{a_k^{1/q_n}: k \in S_n\}\le \lim_n \max \{a_k^{1/Q}: k \in S_n\} =0.
$$
The claim follows by Proposition \ref{proposition:DLidealtall}.
\end{proof}

\begin{thm}\label{DL not generalized density}
Let $\psi=\psi(\bm{\varphi}, q, a, S)$ be a lscsm on $\omega^2$ as in \eqref{eq:definitiondlideal} such that 
$$
\textstyle 
0<\inf_{n \in \omega}\| \omega\|_{\varphi_n} \le \sup_{n \in \omega} \varphi_n(\omega) <\infty
$$
and $\max\{a_k^{1/q_n}: k \in S_n\} \to 0$ as $n\to \infty$. 
Then $\mathrm{Exh}(\psi)$ is not a generalized density ideal.
\end{thm}
\begin{proof}
Suppose for the sake of contradiction that $\mathrm{Exh}(\psi)=\mathrm{Exh}(\sup_n \mu_n)$, where $(\mu_n)$ is a sequence of lscsms on $\omega^2$ with finite pairwise disjoint supports and set $M_n:=\mathrm{supp}(\mu_n)$ for each $n$ (note that $M_n \subseteq \omega^2$). It follows by the standing assumptions that there exists a sequence $(F_n) \in \mathcal{F}_{\mathrm{incr}}$ such that
\begin{equation}\label{eq:construction}
\forall n \in \omega,\quad 
\varphi_n(F_n) \ge \frac{\inf_{k \in \omega}\| \omega\|_{\varphi_k}}{2} 
\,\,\text{ and }\,\,
X_n \cap \bigcup_{\substack{k \in \omega: \\ M_k \cap \,\bigcup_{i\le n-1}X_i \neq \emptyset}} M_k =\emptyset,
\end{equation}
where $X_n:=\{n\} \times F_n$ for each $n$. 

Set $X:=\bigcup_n X_n$. Then $X\notin \mathrm{Exh}(\psi)$. Indeed
\begin{displaymath}
\begin{split}
\|X\|_\psi
&=\inf_{F \in \mathrm{Fin}}\sup_{n \in \omega}(\sum_{k \in S_n} a_k \varphi_k^{q_n}(X_{(k)}\setminus F))^{1/q_n}\\
&\ge \inf_{F \in \mathrm{Fin}}\sup_{n \in \omega} \min_{k \in S_n}\{\varphi_k(X_{(k)}\setminus F)\} \ge \frac{\inf_{k \in \omega}\| \omega\|_{\varphi_k}}{2}>0.
\end{split}
\end{displaymath}

It follows that $X \notin \mathrm{Exh}(\sup_n \mu_n)$, i.e., there exist $\varepsilon>0$ and an increasing sequence $(n_k)$ in $\omega$ such that $\mu_{n_k}(X)>\varepsilon$ for all $k \in \omega$. However, thanks to \eqref{eq:construction}, for each $k$ there exists a unique $m_k \in \omega$ such that $X\cap M_{n_k}\subseteq X_{m_k}$. Therefore 
$\mu_{n_k}(X_{m_k})>\varepsilon$ for all $k \in \omega$. 
Let $M$ be an infinite subset of $\{m_k: k \in \omega\}$ such that $|S_n \cap M|\le 1$ for all $n$. Define $Y:=\bigcup_{m \in M}X_m$ and note that $Y\notin \mathrm{Exh}(\sup_n \mu_n)$. 

Then necessarily $Y\notin \mathrm{Exh}(\psi)$. However, considering that there is at most one $k \in S_n$ such that $Y_{(k)}\neq \emptyset$, we obtain
\begin{displaymath}
\|Y\|_\psi
\le \inf_{F \in [\omega^2]^{<\omega}}\sup_{n \in \omega} \max_{k \in S_n}\{a_k^{1/q_n}\varphi_k((Y\setminus F)_{(k)})\} 
\le \sup_{n \in \omega}\varphi_n(\omega) \limsup_{t\to \infty}\max_{k \in S_t}\{a_k^{1/q_t}\} = 0,
\end{displaymath}
which is the wanted contradiction.
\end{proof}
With the same technique of Corollary \ref{cor:talldlideal}, we obtain (details are omitted):
\begin{cor}\label{cor:talldlidealefssdbsb}
Let $\varphi$ be a strongly-density-like lscsm on $\omega$ such that 
\begin{equation}\label{eq:solecky}
0<\|\omega\|_\varphi\le \varphi(\omega)<\infty. 
\end{equation} 
Moreover, let $\psi=\psi(\bm{\varphi},q,a,S)$ be a lscsm on $\omega^2$ as in \eqref{eq:definitiondlideal} and assume that $\bm{\varphi}$ is the constant sequence $(\varphi)$, $q$ is a bounded sequence, and $\lim_n a_n=0$. 
Then $\mathrm{Exh}(\psi)$ is not a generalized density ideal.
\end{cor}
Note that condition \eqref{eq:solecky} has been already used in the literature, see e.g. \cite[Theorem 3.1]{MR1708146}. 

Putting all together, we have the following:
\begin{thm}\label{thm:sufficienttallexamples}
Let $\psi=\psi(\bm{\varphi}, q, a, S)$ be a lscsm on $\omega^2$ as in \eqref{eq:definitiondlideal} such that: 
\begin{enumerate}[label=\textup{(}\roman*\textup{)}]
\item \label{finaldesnitylike:item1} $\mathrm{Exh}(\varphi_n)$ is tall for each $n \in \omega$\textup{;} 
\item \label{finaldesnitylike:item2} $\lim_n \max\{a_k^{1/q_n}: k \in S_n\}=0$\textup{;} 
\item \label{finaldesnitylike:item3} $\inf_n c(\varphi_n)>0$\textup{;} 
\item \label{finaldesnitylike:item4} $0<\inf_{n}\| \omega\|_{\varphi_n} \le \sup_{n} \varphi_n(\omega) <\infty$\textup{.}
\end{enumerate}
Then $\mathrm{Exh}(\psi)$ is a tall density-like ideal which is not a generalized density ideal.
\end{thm}
\begin{proof}
Thanks to \ref{finaldesnitylike:item4}, we have $\sup_{n,k}\varphi_n(\{k\})\le \sup_n \varphi_n(\omega)<\infty$. The conclusion follows by Proposition \ref{proposition:DLidealtall}, Theorem \ref{DL-density-like}, and Theorem \ref{DL not generalized density}.
\end{proof}

In the case where $\bm{\varphi}$ is a constant sequence and $q$ is bounded, we can simplify the above conditions:
\begin{cor}\label{cor:finalcortalldensitylike}
Let $\varphi$ be a strongly-density-like lscsm on $\omega$ such that 
$\mathrm{Exh}(\varphi)$ is tall and 
satisfies \eqref{eq:solecky}. 
Moreover, let $\psi=\psi(\bm{\varphi}, q, a, S)$ be a lscsm on $\omega^2$ as in \eqref{eq:definitiondlideal} such that $\bm{\varphi}$ is the constant sequence $(\varphi)$, $q$ is bounded, and $\lim_n a_n=0$. 

Then $\mathrm{Exh}(\psi)$ is a tall density-like ideal which is not a generalized density ideal.
\end{cor}
\begin{proof}
It follows by 
Corollary \ref{cor:samevarphinimpliesdensitylike}, 
Corollary \ref{cor:talldlideal}, 
and Corollary \ref{cor:talldlidealefssdbsb}. 
\end{proof}

%
Thus, we answer Question \ref{sec:tallquestionoriiginal}, giving an alternative proof of Theorem \ref{thm::main}. 
\begin{thm}\label{thm:existencetall}
There exists a tall density-like ideal which is not a generalized density ideal.
\end{thm}
\begin{proof}
Let $\mathcal{I}_d$ be the ideal of density zero sets, which is a tall ideal. Thanks to \cite[Example 1.2.3.(d), Theorem 1.13.3(a), and Lemma 1.13.9.(Z3)]{MR1711328}, there exists a sequence $(\mu_n)$ of probability measures with finite pairwise disjoint supports such that $\mathcal{I}_d=\mathrm{Exh}(\varphi)$, where $\varphi:=\sup_n \mu_n$. In particular, $\varphi(\omega)=\|\omega\|_\varphi=1$. The claim follows by Corollary \ref{cor:finalcortalldensitylike}.
%
%
%
%
%
%
%
%
%
\end{proof}

With the same spirit of Question \ref{q:questiondensitylikejhsjshjdf}, we ask:
\begin{question}\label{q:infinitelymanynonisomtallexmaples}
How many pairwise nonisomorphic tall density-like ideals which are not generalized density ideals are there?
\end{question}

We will show that, as in Theorem \ref{cor:2omega}, there is a family of $2^\omega$ such ideals. To this aim, we need a preliminary lemma.
\begin{lem}\label{lem:disjoinmad}
There exists a family 
$\mathscr{A}$ 
of $2^\omega$ subsets of $\omega^2$ such that 
\begin{equation}\label{eq:madalmost}
\forall (A,A^\prime) \in [\mathscr{A}]^2, \forall k \in \omega,\quad A_{(k)} \notin \mathrm{Fin} 
\,\,\,\text{ and }\,\,\,
A \cap A^\prime \in \mathrm{Fin}.
\end{equation}
\end{lem}
\begin{proof}
It is known that there exists a family $\mathscr{B}$ of $2^\omega$ subsets of $\omega$ such that
$$
\forall (B,B^\prime) \in [\mathscr{B}]^2, \quad B \notin \mathrm{Fin} 
\,\,\,\text{ and }\,\,\,
B \cap B^\prime \in \mathrm{Fin},
$$
see e.g. \cite[Lemma 2.5.3]{MR3526021}. 
Then, it is sufficient to see that $\{\,\bigcup_{n\in \omega} \{n\}\times (B\setminus n): B \in \mathscr{B}\}$ satisfies \eqref{eq:madalmost}. 
\end{proof}

\begin{thm}\label{thm:mainsecondparttall}
There are $2^\omega$ nonpathological and pairwise nonisomorphic tall density-like ideals which are not generalized density ideals.
\end{thm}
\begin{proof}
Let $h: \omega^2 \to \omega$ be the bijection defined in \eqref{eq:bijectionh}. 
Let also $\mathscr{M}=\{M_z: z \in \omega^2\}$ be a partition of $\omega^2$ into nonempty finite sets such that $M_{(n,m)}\subseteq \{n\}\times \omega$ for all $n,m\in\omega$ and
\begin{equation}\label{eq:assumptioncardinality}
\forall n \in \omega, \quad m_{h^{-1}(n+1)}\ge (n+2)\sum_{i\le n}m_{h^{-1}(i)}, 
\end{equation}
where $m_z:=|M_z|$. 
Moreover, for each $(n,m) \in \omega^2$, let $\mu_{(n,m)}$ be the uniform probability measure given by $\mu_{(n,m)}(X)= |(\{n\}\times X)\cap M_{(n,m)}|/m_{(n,m)}$ for all $X\subseteq \omega$. 

Fix $(S_n) \in \mathcal{F}_{\mathrm{incr}}$ such that $\lim_n |S_n|=\infty$ and let $\mathscr{A}$ be a family of $2^\omega$ subsets of $\omega^2$ which satisfies \eqref{eq:madalmost} (existing by Lemma \ref{lem:disjoinmad}). For each $A \in \mathscr{A}$, let $\psi_A$ be the lscsm on $\omega^2$ defined by
$$
\psi_A:=\sup_{n\in \omega} \sum_{k \in S_n}\frac{1}{|S_n|}\,\varphi_{A,k},
\,\,\,\text{ where }\,\,\,\varphi_{A,k}:=\sup_{t \in A_{(k)}} \mu_{(k,t)}
$$
for all $k \in \omega$. 
It follows that, for each $A \in \mathscr{A}$, the lscsm $\psi_A$ is of the type \eqref{eq:definitiondlideal}, where $q_n=1$ for all $n$ and $a_k=1/|S_i|$ whenever $k \in S_i$; hence $\lim_n a_n=0$. In addition, for each $A \in \mathscr{A}$ and $k \in \omega$, the ideal $\mathrm{Exh}(\varphi_{A,k})$ is tall and $\|\omega\|_{\varphi_{A,k}}=\varphi_{A,k}(\omega)=1$. Lastly, thanks to Remark \ref{rem-strongly-density-like}, each $\varphi_{A,k}$ is strongly-density-like with any witnessing constant $c(\varphi_{A,k})<1$. In particular, $\inf_{A,k}c(\varphi_{A,k})\ge \nicefrac{1}{2} >0$. Therefore, by Theorem \ref{thm:sufficienttallexamples}, each $\mathrm{Exh}(\psi_A)$ is a tall density-like ideal which is not a generalized density ideal. 
Also, by Remark \ref{rmk:nonpathologicaltall}, each $\mathrm{Exh}(\psi_A)$ is nonpathological.

At this point, we claim that, for all distinct $A,A^\prime \in \mathscr{A}$, the ideals $\mathrm{Exh}(\psi_A)$ and $\mathrm{Exh}(\psi_{A^\prime})$ are not isomorphic. To this aim, fix distinct $A,A^\prime \in \mathscr{A}$ and suppose for the sake of contradiction that $\mathrm{Exh}(\psi_A)$ and $\mathrm{Exh}(\psi_{A^\prime})$ are isomorphic, witnessed by a bijection $f:\omega^2 \to \omega^2$. Let $(x_n)$ be the enumeration of the infinite set $A\setminus A^\prime$ such that the sequence $(h(x_n))$ is increasing. Then, pick a sequence $(F_n) \in \mathcal{F}_{\mathrm{disj}}$ such that
$$
\forall n \in \omega,\quad  
F_n\subseteq M_{x_n}, 
\,\,\, |F_n|=\left\lfloor \frac{m_{x_n}}{2}\right\rfloor, 
\,\,\,\text{ and }\,\,\,
F_{n+1}\cap \bigcup_{\substack{z \in \omega^2: \\ h(z)<h(x_{n+1})}}f[M_z]=\emptyset.
$$
Note that this is really possible: indeed, letting $U$ be the latter union, it follows by \eqref{eq:assumptioncardinality} that 
$$
|U|=
\sum_{h(z)<h(x_{n+1})}m_z=
\sum_{i\le h(x_{n+1})-1}m_{h^{-1}(i)}
\le \frac{1}{2}\,m_{x_{n+1}}.
$$

Set $F:=\bigcup_n F_n$. It follows by construction that $\|F\|_{\varphi_{A,k}}=\nicefrac{1}{2}$ for all $k \in \omega$, hence $F \notin \mathrm{Exh}(\psi_A)$. On the other hand, we obtain by \eqref{eq:assumptioncardinality} that
\begin{displaymath}
\begin{split}
\forall (i,j) \in A^\prime \setminus A,\quad \mu_{(i,j)}((f^{-1}[F])_{(i)})
&=\frac{|F\cap f[M_{(i,j)}]|}{m_{(i,j)}} \\
&\le \frac{\sum_{n\in\{k:\ h(x_k)<h((i,j))\}}|F_n|}{m_{(i,j)}} \\
&\le \frac{\sum_{k\le h((i,j))-1}m_{h^{-1}(k)}}{m_{h^{-1}(h((i,j)))}} \le \frac{1}{h((i,j))+1},
\end{split}
\end{displaymath}
which implies that $f^{-1}[F] \in \mathrm{Exh}(\psi_{A^\prime})$. This contradiction concludes the proof.
\end{proof}

\section{Characterization of Generalized Density Ideals}\label{sec:conversegeneralizeddensitylike}


In this section, we provide a characterization of generalized density ideals which resembles the one of density-like ideal given in Definition \ref{def:densitylike}. This provides sufficient conditions for a density-like ideal to be necessarily a generalized density ideal, 


Let $\mathscr{H}$ be the set of strictly increasing sequences in $\omega$. Then, given a lscsm $\varphi$ and a real $\varepsilon>0$, define
$$
K_{s, F}:=\{(k_n) \in \mathscr{H}: \forall n\in \omega, \exists m \in \omega, \max F_{k_n}\le s_m < \min F_{k_{n+1}}\}
$$ 
for all $s=(s_n) \in \mathscr{H}$ and $F=(F_n) \in \mathcal{F}_{\mathrm{disj}} \cap \mathcal{G}_{\varphi, \varepsilon}$ (note that $K_{s, F}\neq \emptyset$). 

\begin{defi}\label{def:dweak}
A lscsm $\varphi$ on $\omega$ satisfies \textbf{condition} $\bm{\Dweak}$ if for all $\varepsilon>0$ there exist $\delta>0$ and a sequence $s \in \mathscr{H}$ such that, if $F=(F_n) \in\mathcal{F}_{\mathrm{incr}} \cap \mathcal{G}_{\varphi,\delta}$ and $k\in K_{s, F}$, then $\varphi(\bigcup_{n} F_{k_n})<\varepsilon$.
\end{defi}

If the sequence $s \in \mathscr{H}$ can be chosen uniformly in $\varepsilon>0$, we have the following:
\begin{defi}\label{def:dweak}
A lscsm $\varphi$ on $\omega$ satisfies \textbf{condition} $\bm{\Dstrong}$ if there exists a sequence $s \in \mathscr{H}$ for which for all $\varepsilon>0$ there exists $\delta>0$ such that, if $F=(F_n) \in\mathcal{F}_{\mathrm{incr}} \cap \mathcal{G}_{\varphi,\delta}$ and $k\in K_{s, F}$, then $\varphi(\bigcup_{n} F_{k_n})<\varepsilon$.
\end{defi}

It is clear that every lscsm $\varphi$ satisfying condition $\Dstrong$ satisfies also condition $\Dweak$. 
%
Thus, 
we state the main result of this section.
\begin{thm}\label{char}
Let $\mathcal{I}$ be an ideal. Then the following are equivalent:
\begin{enumerate}[label={\rm (\textsc{A}\arabic{*})}]
\item \label{item:abc1} $\I$ is a generalized density ideal\textup{;}
\item \label{item:abc2} every lscsm $\varphi$ such that $\I=\Exh(\varphi)$ satisfies condition $\Dstrong$\textup{;}
\item \label{item:abc3} every lscsm $\varphi$ such that $\I=\Exh(\varphi)$ satisfies condition $\Dweak$\textup{;}
\item \label{item:abc4} $\I=\Exh(\varphi)$ for some lscsm $\varphi$ satisfying condition $\Dstrong$\textup{;}
\item \label{item:abc5} $\I=\Exh(\varphi)$ for some lscsm $\varphi$ satisfying condition $\Dweak$\textup{.}
\end{enumerate}
\end{thm}


The proof is divided in some intermediate steps. 
\begin{lem}\label{lem:firstimplicationlastsection}
Let $\varphi$ be a lscsm and assume that $\Exh(\varphi)$ is a generalized density ideal. Then $\varphi$ satisfies condition $\Dstrong$. 
\end{lem}
\begin{proof}
Let us suppose for the sake of contradiction that $\varphi$ does not satisfy condition $\Dstrong$. Thanks to Proposition \ref{prop:supportgeneralizeddensityideal}, we can suppose without loss of generality that there exist a sequence $(\mu_n)$ of submeasures and a sequence $(S_n) \in \mathcal{F}_{\mathrm{int}}$ of consecutive intervals of $\omega$ such that $\varphi=\sup_n\mu_n$ and $S_n=\mathrm{supp}(\mu_n)$ for all $n \in \omega$. 
In particular, $\mathrm{Exh}(\varphi)=\{A\subseteq \omega: \lim_n \mu_n(A)=0\}$ and $\mu_n(\omega)\not\to 0$. 

Define $s \in \mathscr{H}$ by $s_n:=\max S_n$ for all $n \in \omega$. Since $\varphi$ does not satisfy condition $\Dstrong$, there exists $\varepsilon>0$ such that for all nonzero  $m\in\omega$ there are $F^m=(F^m_n)\in\mathcal{F}_{\mathrm{incr}} \cap \mathcal{G}_{\varphi,\frac{\varepsilon}{2^m}}$ and $k^m=(k^m_n)\in K_{s,F^m}$ for which $\varphi(\bigcup_{n} F^m_{k_n^m})\geq\varepsilon$. Since $\varphi$ is a lscsm, for each $m$ there exists $\ell_m \in \omega$ such that $\varphi(\bigcup_{n\leq \ell_m} F^m_{k^m_n})\geq\nicefrac{\varepsilon}{2}$. 

Set $G_m:=\bigcup_{n\leq \ell_m} F^m_{k^m_n}$ and note that $G:=\bigcup_m G_m$ does not belong to $\mathrm{Exh}(\varphi)$. 
To this aim, fix a nonzero $j\in \omega$. 
Since $(k^j_n)\in K_{s,F^j}$, there are at most $j$ many sets $F^j_{k_n^j}$ which have nonempty intersection with the set $s_j+1$ and each of them has $\varphi$-value smaller than $\varepsilon/2^{j}$. Thus 
\begin{displaymath}
\begin{split}
\varphi(G\setminus (s_j+1)) \ge \varphi(G_j\setminus (s_j+1))\ge \varphi(G_j)-\varphi(G_j \cap (s_j+1))\ge \varepsilon-j \frac{\varepsilon}{2^j}\ge \frac{\varepsilon}{2}.
\end{split}
\end{displaymath}
Therefore $\|G\|_\varphi \ge \nicefrac{\varepsilon}{2}>0$. In particular, there exists a sequence $(j_i) \in \mathscr{H}$ such that $\mu_{j_i}(G) > \nicefrac{\varepsilon}{3}$ for all $i\in \omega$. Passing eventually to a subsequence, we can assume without loss of generality that 
\begin{equation}\label{eq:inequalitysI}
\forall i \in \omega, \quad s_{j_i}>\max G_{i+1}.
\end{equation}

At this point, define $X:=G\cap \bigcup_i S_{j_i}$. Then by construction $\|X\|_\varphi \ge \nicefrac{\varepsilon}{3}>0$, so that $X\notin \mathrm{Exh}(\varphi)$. On the other hand, we will show that $\lim_n \mu_n(X)=0$, reaching a contradiction. Taking into account \eqref{eq:inequalitysI}, note that $S_{j_i}\cap X= S_{j_i} \cap \bigcup_{m>i}G_m$ for all $i$. Moreover, recall that, for all $i,m \in \omega$, there exists at most one $n$ such that $F^m_{k^m_n}\cap S_{j_i} \neq \emptyset$. 
It follows that
$$
\forall i \in \omega,\quad \varphi(S_{j_i}\cap X)=\mu_{j_i}(\bigcup_{m>i}G_m)\le \sum_{m>i}\mu_{j_i}(G_m) \le \sum_{m>i}\frac{\varepsilon}{2^m}=\frac{\varepsilon}{2^i}.
$$
To conclude, we obtain that
$$
\forall t \in \omega, \quad \varphi(X\setminus \bigcup_{i\le t}S_{j_i})=\varphi(X\cap \bigcup_{i> t}S_{j_i})\le \sum_{i>t}\varphi(S_{j_i} \cap X)\le \sum_{i>t}\frac{\varepsilon}{2^i}= \frac{\varepsilon}{2^t},
$$
which tends to $0$ as $t\to \infty$. Hence $X \in \mathrm{Exh}(\varphi)$, which is the wanted contradiction.
\end{proof}

Now, we show that condition $\Dweak$ implies (a variant of) condition $\Dstrong$. 
\begin{lem}\label{lem3.1}
Let $\varphi$ be a lscsm which satisfies condition $\Dweak$. Then there is a lscsm $\nu$ and a sequence $s \in \mathscr{H}$ such that for every $\varepsilon>0$ there is $\delta>0$ for which, if $F \in\mathcal{F}_{\mathrm{incr}} \cap \mathcal{G}_{\nu,\delta}$ and $k \in K_{s,F}$, then $\nu(\bigcup_{n} F_{k_{2n}})<\varepsilon$, and $\Exh(\varphi)=\Exh(\nu)$.
\end{lem}
\begin{proof}
Let $(\varepsilon_k)$ be a strictly decreasing sequence such that $\lim_k \varepsilon_k=0$. 
Then, for each $k$, there are $\delta_k>0$ and a sequence $s^k=(s_n^k) \in \mathscr{H}$ such that, if $F \in\mathcal{F}_{\mathrm{incr}} \cap \mathcal{G}_{\varphi,\delta_k}$ and $k \in K_{s^k, F}$, then $\varphi(\bigcup_{n} F_{k_n})<\varepsilon_k$. Without loss of generality, we can assume that $\delta_{k+1}<\delta_k<\varepsilon_k$. 
Let us define $s=(s_n) \in \mathscr{H}$ as follows: $s_0:=s_0^0$ and, for each $n \in \omega$, let $s_{n+1}$ be such that for all $m\leq n+1$ there is $\ell \in \omega$ such that $s_n\leq s^m_\ell<s_{n+1}$. Then, set $S_0:=[0,s_0]$ and $S_{n+1}:=(s_{n},s_{n+1}]$ for all $n\in \omega$. 

Also, let $\psi$ be the lscsm defined by $\psi(\emptyset)=0$ and, for each nonempty $A\subseteq \omega$, $\psi(A):=\delta_k$, where $k$ is the minimal integer such that $A\cap S_k\neq\emptyset$. 
At this point, set $\nu:=\max\{\varphi,\psi\}$. Then $\nu$ is a lscsm such that $\Exh(\nu)=\Exh(\varphi)$. Indeed, on the one hand, $\nu\geq\varphi$, hence $\Exh(\nu)\subseteq\Exh(\varphi)$. On the other hand, fix $A\in\Exh(\varphi)$ and $\varepsilon>0$, hence there is $n_0\in \omega$ such that $\mu(A\setminus n)<\varepsilon$ for all $n\geq n_0$. Also, there is $n_1\in \omega$ such that $\delta_n<\varepsilon$ for all $n\geq n_1$. Thus, for each $n$ such that $n\geq n_0$ and $\bigcup_{i<n_1}S_i\subseteq n$ we have $\nu(A\setminus n)<\varepsilon$. Therefore $A\in\Exh(\nu)$, which proves the opposite inclusion $\Exh(\varphi)\subseteq\Exh(\nu)$.

Lastly, we show that $\nu$ satisfies the condition in the statement. 
Fix $\varepsilon>0$ and let $m$ be the minimal integer such that $\varepsilon_m\leq \varepsilon$. We claim that $\delta:=\delta_m$ witnesses this condition. Fix $(F_n) \in\mathcal{F}_{\mathrm{incr}} \cap \mathcal{G}_{\nu,\delta}$, $k \in K_{s,F}$, and $j\in\omega$. Then there exist $\ell^\prime,\ell^{\prime\prime} \in \omega$ such that
$$
\max F_{k_{2j}}\leq s_{\ell^\prime}<\min F_{k_{2j+1}}\leq\max F_{k_{2j+1}}\leq s_{\ell^{\prime\prime}}<\min F_{k_{2(j+1)}}.
$$
Note that $m<\ell^\prime<\ell^{\prime\prime}$, where the former inequality follows by the fact that $F_n \cap \bigcup_{i\le m}S_i =\emptyset$ for all $n$ (by the definition of $\nu$ and the hypothesis $(F_n) \in\mathcal{F}_{\mathrm{incr}} \cap \mathcal{G}_{\nu,\delta}$) and the latter since $s\in \mathscr{H}$. Thus, there exists $\ell \in \omega$ such that 
$$
s_{\ell^\prime}\leq s^m_\ell<s_{\ell^\prime+1}\leq s_{\ell^{\prime\prime}}.
$$ 
It follows that $\max F_{k_{2j}}\leq s^m_{\ell}<\min F_{k_{2(j+1)}}$, so that $(k_{2n})\in K_{s^m, F}$. 
Therefore $\varphi(\bigcup_{n} F_{k_{2n}})<\varepsilon_m\leq\varepsilon$. 
It is also easy to see that $\psi(\bigcup_{n} F_{k_{2n}})\leq\delta=\delta_m<\varepsilon_m\leq\varepsilon$. 
Putting all together, we conclude that $\nu(\bigcup_{n} F_{k_{2n}})<\varepsilon$.
\end{proof}

\begin{lem}\label{lem3.2}
Let $\nu$ be a lscsm as in Lemma 	\ref{lem3.1}. 
Then $\Exh(\nu)$ is a generalized density ideal.
\end{lem}
\begin{proof}
Define $S_0:=[0,s_0]$ and $S_{n+1}:=(s_n,s_{n+1}]$ for all $n\in\omega$. 
Let $\bm{\mu}=(\mu_n)$ be the sequence of lscsms defined by
$$
\forall n \in \omega, \forall A\subseteq \omega,\quad \mu_n(A):=\nu(A\cap S_n).
$$
We claim that $\Exh(\varphi_{\bm{\mu}})=\Exh(\nu)$, where $\varphi_{\bm{\mu}}:=\sup_n \mu_n$. 

It is clear that $\varphi_{\bm{\mu}} \le \nu$, hence $\Exh(\nu)\subseteq \Exh(\varphi_{\bm{\mu}})$. 
Conversely, fix $A\in\Exh(\varphi_{\bm{\mu}})$ and $\varepsilon>0$, hence there is $\delta>0$ such that, if $F=(F_n)\in\mathcal{F}_{\mathrm{incr}} \cap \mathcal{G}_{\nu,\delta}$ and $k\in K_{s,F}$, then $\nu(\bigcup_{n} F_{k_{2n}})<\nicefrac{\varepsilon}{2}$. 
There exists $n_0\in \omega$ such that $\mu_n(A)<\delta$ for all $n\geq n_0$. 
Define $F_n:=A\cap S_{n+n_0}$ for all $n\in\omega$. 
Then $\nu(F_n)=\mu_{n+n_0}(A)<\delta$ for all $n \in \omega$. 
Thus $(F_n)\in\mathcal{F}_{\mathrm{incr}} \cap \mathcal{G}_{\nu,\delta}$ and for each $k \in K_{s,F}$ we have $\nu(\bigcup_{n} F_{k_{2n}})<\nicefrac{\varepsilon}{2}$. 
Note the sequences $(n)$ and $(n+1)$ belong to $K_{s,F}$, so that 
$\nu(\bigcup_{n} F_{2n})<\nicefrac{\varepsilon}{2}$ and $\nu(\bigcup_{n} F_{2n+1})<\nicefrac{\varepsilon}{2}$. 
Define $m:=\min S_{n_0}$. Then for each $m\geq m_0$ we have
$$
\nu(A\setminus m) \le \nu(A\setminus m_0)=\nu(\bigcup_{n \in \omega} F_n)
\le \nu(\bigcup_{n\in \omega} F_{2n})+\nu(\bigcup_{n\in \omega} F_{2n+1})
<\varepsilon.
$$
We conclude that $A\in\Exh(\nu)$, therefore $\Exh(\varphi_{\bm{\mu}})\subseteq \Exh(\nu)$.
\end{proof}

We are finally ready to prove Theorem \ref{char}, cf. Figure \ref{fig:prooflastsection} below.

\begin{proof}[Proof of Theorem \ref{char}]
\ref{item:abc1} $\implies$ \ref{item:abc2} follows by Lemma \ref{lem:firstimplicationlastsection}. 
The implications \ref{item:abc2} $\implies$ \ref{item:abc3} $\implies$ \ref{item:abc5} and \ref{item:abc2} $\implies$ \ref{item:abc4} $\implies$ \ref{item:abc5} are obvious. 
Lastly, \ref{item:abc5} $\implies$ \ref{item:abc1} follows by Lemma \ref{lem3.1} and Lemma \ref{lem3.2}. 
\end{proof}

\bigskip

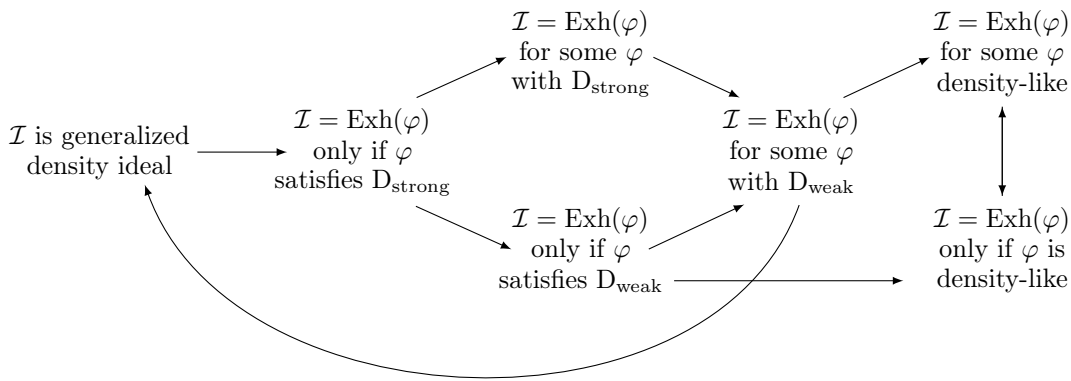
\begin{figure}[!htb]
\centering
\begin{tikzpicture}
[scale=1]
\node (a1) at (0,0){{\footnotesize $\mathcal{I}$ is generalized}};
\node (a1b) at (0,-.4){{\footnotesize density ideal}};

\node (a2) at (3.5,0.2){{\footnotesize $\mathcal{I}=\mathrm{Exh}(\varphi)$}};
\node (a2b) at (3.5,-.2){{\footnotesize only if $\varphi$}};
\node (a2c) at (3.5,-.6){{\footnotesize satisfies $\Dstrong$}};

\node (a3) at (6.4,0.2+1.3){{\footnotesize $\mathcal{I}=\mathrm{Exh}(\varphi)$}};
\node (a3b) at (6.4,-.2+1.3){{\footnotesize for some $\varphi$}};
\node (a3c) at (6.4,-.6+1.3){{\footnotesize with $\Dstrong$}};

\node (a4) at (6.4,0.2-1.3){{\footnotesize $\mathcal{I}=\mathrm{Exh}(\varphi)$}};
\node (a4b) at (6.4,-.2-1.3){{\footnotesize only if $\varphi$}};
\node (a4c) at (6.4,-.6-1.3){{\footnotesize satisfies $\Dweak$}};

\node (a5) at (9.2,0.2){{\footnotesize $\mathcal{I}=\mathrm{Exh}(\varphi)$}};
\node (a5b) at (9.2,-.2){{\footnotesize for some $\varphi$}};
\node (a5c) at (9.2,-.6){{\footnotesize with $\Dweak$}};

\node (a6) at (12,0.2+1.3){{\footnotesize $\mathcal{I}=\mathrm{Exh}(\varphi)$}};
\node (a6b) at (12,-.2+1.3){{\footnotesize for some $\varphi$}};
\node (a6c) at (12,-.6+1.3){{\footnotesize density-like}};

\node (a7) at (12,0.2-1.3){{\footnotesize $\mathcal{I}=\mathrm{Exh}(\varphi)$}};
\node (a7b) at (12,-.2-1.3){{\footnotesize only if $\varphi$ is}};
\node (a7c) at (12,-.6-1.3){{\footnotesize density-like}};

\draw [-latex, thin, shorten >= .2cm, shorten <= 1.3cm] (0,-.2)--(a2b);

\draw [-latex, thin,  shorten >= .3cm] (a2)--(a3);
\draw [-latex, thin,  shorten <= .3cm] (a3)--(a5);
\draw [-latex, thin, shorten >= .3cm] (a5)--(a6);

\draw [-latex, thin,  shorten >= .3cm] (a2c)--(a4c);
\draw [-latex, thin,  shorten <= .3cm] (a4c)--(a5c);
\draw [-latex, thin, shorten >= .3cm] (a4c)--(a7c);

\draw [-latex, thin] (a6c)--(a7);
\draw [-latex, thin] (a7)--(a6c);

\draw[-latex] (9.3,-.9) to[out=250, in=290](.6,-.63);
\end{tikzpicture}
\caption{Relationship between generalized density ideals and density-like ideals, assuming $\mathcal{I}$ is an analytic P-ideal.}
\label{fig:prooflastsection}
\end{figure}

%


\section{Concluding Remarks}\label{sec:concludingrmk}

Differently from the case of density-like lscsms, if $\varphi$ and $\psi$ are two lscsms such that $\mathrm{Exh}(\varphi)=\mathrm{Exh}(\psi)$ and $\varphi$ is strongly-density-like, then $\psi$ is \emph{not} necessarily strongly-density-like. 

Indeed, let $\varphi$ be the strongly-density-like lscsm defined in Example \ref{ex-strongly-density-like}. Then, with the same notations, it is easily seen that $A\in \mathrm{Exh}(\varphi)$ if and only if $A\cap X_n\in\Fin$ for all $n \in \omega$, hence $\Exh(\varphi)$ is isomorphic to $\emptyset\times\Fin$. However, $\emptyset\times\Fin$ is a generalized density ideal and, thanks to Remark \ref{rem-strongly-density-like}, there exists a strongly-density-like lscsm $\psi$ such that $\Exh(\varphi)=\Exh(\psi)$. We conclude with an open question.

\begin{question}
Does there exist a density-like ideal $\I$ such that $\I\neq\Exh(\varphi)$ for each strongly-density-like lscsm $\varphi$?
\end{question}


\subsection*{Acknowledgments} 
The authors are grateful to an anonymous reviewer for his careful reading of the manuscript and constructive suggestions, and to Jacek Tryba (University of Gda{\'n}sk, PL) for several useful comments and for pointing out reference \cite{MR2078923}. 
P.L. thanks PRIN 2017 (grant 2017CY2NCA)
for financial support.


\bibliographystyle{amsplain}

\end{document}